\documentclass[11pt]{article}
\usepackage{amsfonts}
\usepackage{amssymb}
\usepackage{amsthm}
\usepackage{amsmath}
\usepackage{graphicx}
\usepackage{empheq}
\usepackage{indentfirst}
\usepackage{cite}
\usepackage{mathrsfs}
\usepackage{cases}
\usepackage{graphics}
\usepackage{xcolor}
\usepackage{amsmath,bm}
\newtheorem{theorem}{\textbf{Theorem}}[section]
\newtheorem{lemma}{\textbf{Lemma}}[section]
\newtheorem{proposition}{\textbf{Proposition}}[section]
\newtheorem{corollary}{\textbf{Corollary}}[section]
\newtheorem{remark}{\textbf{Remark}}[section]
\newtheorem{definition}{\textbf{Definition}}[section]

\allowdisplaybreaks[4]

\def\be{\begin{equation}}
\def\ee{\end{equation}}
\def\bea{\begin{eqnarray}}
\def\eea{\end{eqnarray}}
\def\bt{\begin{theorem}}
\def\et{\end{theorem}}
\def\bl{\begin{lemma}}
\def\el{\end{lemma}}
\def\br{\begin{remark}}
\def\er{\end{remark}}
\def\bp{\begin{proposition}}
\def\ep{\end{proposition}}
\def\bc{\begin{corollary}}
\def\ec{\end{corollary}}
\def\bd{\begin{definition}}
\def\ed{\end{definition}}
\begin{document}

\title{On the Viscous Cahn--Hilliard--Oono System with Chemotaxis and Singular Potential}

\author{
{Jingning He}
\footnote{School of Mathematical Sciences,
Fudan University, Handan Road 220, Shanghai 200433, China. Email:  \textit{jingninghe2020@gmail.com}
}}

\date{\today}

\maketitle


\begin{abstract}
\noindent We analyze a diffuse interface model that couples a viscous Cahn--Hilliard equation for the phase variable with a diffusion-reaction equation for the nutrient concentration. The system under consideration also takes into account some important mechanisms like chemotaxis, active transport as well as nonlocal interaction of Oono's type. When the spatial dimension is three, we prove the existence and uniqueness of global weak solutions to the model with singular potentials including the physically relevant logarithmic potential. Then we obtain some regularity properties of the weak solutions when $t>0$. In particular, with the aid of the viscous term, we prove the so-called instantaneous separation property of the phase variable such that it stays away from the pure states $\pm1$ as long as $t>0$. Furthermore, we study long-time behavior of the system, by proving the existence of a global attractor and characterizing its $\omega$-limit set.\medskip


\end{abstract}

\section{Introduction}
\setcounter{equation}{0}
\noindent
In this paper, we consider the following system of partial differential equations
\begin{subequations}
\begin{alignat}{3}
&\partial_t \varphi=\Delta \mu-\alpha(\varphi-c_0),\ &\textrm{in}\ \Omega\times(0,+\infty),\label{f1.a} \\
&\mu=A\varPsi'(\varphi)-B\Delta \varphi-\chi \sigma+\epsilon\partial_t \varphi, \quad\quad &\textrm{in}\ \Omega\times(0,+\infty),\label{f4.d} \\
&\partial_t \sigma= \Delta (\sigma+\chi(1-\varphi)),\ &\textrm{in}\ \Omega\times(0,+\infty), \label{f2.b}
\end{alignat}
\end{subequations}
subject to the boundary conditions
\begin{alignat}{3}
{\partial}_{\bm{n}}\varphi={\partial}_{\bm{n}}\mu={\partial}_{\bm{n}}\sigma=0,\qquad\qquad &\textrm{on}& \   \partial\Omega\times(0,+\infty),
\label{boundary}
\end{alignat}
as well as initial conditions
\begin{alignat}{3}
&\varphi(0)=\varphi_{0},\ \ \sigma(0)=\sigma_{0}, \qquad &\textrm{in}&\ \Omega.
\label{ini0}
\end{alignat}
Here, $\Omega \subset\mathbb{R}^3$ is a bounded domain with smooth boundary $\partial\Omega$ and $\bm{n}=\bm{n}(x)$ denotes the unit outward normal vector on $\partial\Omega$.

The system \eqref{f1.a}--\eqref{f2.b} can be viewed as a simplified, fluid-free version of the general thermodynamically consistent diffuse
interface model derived in \cite{LW} for a two-phase incompressible fluid mixture with a chemical
species subject to some important mechanisms like diffusion, chemotaxis interaction and active transport.
The order parameter $\varphi$ denotes the difference of volume fractions for the two components, while $\sigma$ standards for the concentration of the nutrient. The function $\mu$ defined by \eqref{f4.d} is regarded as the chemical potential associated to $(\varphi, \sigma)$. In \eqref{f4.d}, $A, B$ are two positive constants related to the surface tension and the thickness of the diffuse interface, respectively. The bulk free energy $\varPsi$ considered in this paper enjoys a double well structure that leads to the separation phenomenon of the binary mixture. A typical form is given by
\be
\varPsi (r)=\frac{\theta}{2}[(1-r)\ln(1-r)+(1+r)\ln(1+r)]+\frac{\theta_{0}}{2}(1-r^2),\quad \forall\, r\in(-1,1),
\label{pot}
\ee
with $0<\theta<\theta_{0}$ (see e.g., \cite{CH,CMZ}).  Comparing with the classical quartic potential $\varPsi_{\mathrm{reg}}(r)=(r^2-1)^2$, the singular nature of its derivative $\varPsi'$ at the pure phases $\pm 1$ guarantees that the phase variable $\varphi$ lies in the physical range $[-1,1]$ (see \cite{A2007,DD,MZ04}).
The nontrivial coupling between the Cahn--Hilliard equation \eqref{f1.a}--\eqref{f4.d} and the diffusion equation \eqref{f2.b} for the nutrient is characterized  by the constant $\chi$, which models some specific mechanisms such as chemotaxis/active transport in the context of tumor growth modeling (see, e.g., \cite{GLSS, GL17e}).
The Cahn--Hilliard equation \eqref{f1.a} also involves some nonlocal interaction that is given by Oono's type $-\alpha(\varphi-c_0)$ for the sake of simplicity (cf. e.g., \cite{GG,Mi11}), where $\alpha\geq 0$, $c_0\in(-1,1)$. Besides, it contains a viscous term $\epsilon\partial_t\varphi$ with $\epsilon>0$, which represents possible influence of the internal microforces (see, e.g. \cite{G,N}). In this sense, one may call \eqref{f1.a}--\eqref{f4.d} (neglecting the term $\chi\sigma)$ a viscous Cahn--Hilliard--Oono system.

The Cahn--Hilliard type models have been employed as an efficient mathematical tool for the study on dynamics of binary mixtures, in particular, recently for the tumour growth modelling \cite{CLLW,HvO,OHP}.
Concerning the mathematical analysis of the Cahn--Hilliard equation and its variants, we refer to \cite{A2007,CMZ,CGH,CGW14,CRW,DD,Fa15,FGR,GL17,GL17e,GLR,GLS2021,GGL2010,Mi11,MT,MRS,MZ04} and the references cited therein. Our aim in this paper is to perform a first step study on the simplified model \eqref{f1.a}--\eqref{f2.b} that still maintains some interesting features, e.g., the effects of chemotaxis and active transport associated with the nutrient, and certain nonlocal interactions between the two components themselves.

First, under suitable assumptions on the singular potential function $\varPsi$ and coefficients of the system, we show the existence and uniqueness of global weak solutions to problem \eqref{f1.a}--\eqref{ini0} on the whole interval $[0,+\infty)$ (see Theorem \ref{main1}). The proof is based on a suitable Galerkin approximation that mainly follows the argument in \cite{H} for a more general system with fluid interaction. Thanks to the singular potential, we are allowed to remove certain restricted assumptions on the coefficients $A$ and $\chi$ when a regular potential was adopted (cf. \cite{GL17e,LW}). We recall that the existence of global weak solutions to the Cahn--Hilliard--Oono system with a singular potential has been proven in \cite{MT,GG}.
Our result extends the previous works to the case with chemotaxis and active transport, i.e., the coupling with equation \eqref{f2.b} and $\chi\neq 0$. For well-posedness results on more general system with further mechanics like tumour proliferation, apoptosis and the nutrient consumption, subject to different types of boundary conditions, we refer to \cite{GL17,GL17e} etc.

Next, we pay our attention to the regularity of global weak solutions to problem \eqref{f1.a}--\eqref{ini0}. After obtaining the existence and uniqueness of global strong solutions (see Theorem \ref{main}), we can take advantage of the parabolic nature of the evolution system and show the instantaneous regularizing effect of weak solutions for $t>0$ (see Corollary \ref{cr}). In particular, we prove the so-called instantaneous strict separation property, which guarantees that the weak solution stays away from pure phases $\pm 1$ for all $t>0$ and the separation is uniform when $t\geq \eta$, for an arbitrary but fixed $\eta>0$.

The property of separation from pure states plays an important role in the study of the Cahn--Hilliard type equations, since the singular potential can be regarded as a globally Lipschitz function and thus further regularity of solutions can be gained (see e.g., \cite{A2007,CMZ,MZ04,GG}). When the logarithmic potential \eqref{pot} is considered, the strict separation property for the Cahn--Hilliard equation in dimension two has been proven in \cite{MZ04}. Later on, the authors of \cite{GG} extend the result to the Cahn--Hilliard--Oono equation by using an alternative
approach (which can be generalized to more complicated systems with fluid interactions, see e.g., \cite{GGW,GMT,H1}). However, the situation is less satisfactory when the spatial dimension is three, because the singularity of the logarithmic potential seems not strong enough. On one hand, it was shown in \cite{A2007} that a weak solution of the Cahn--Hilliard equation will stay eventually away from the pure states for sufficiently large time. On the other hand, there were two possible ways in the literature to recover the instantaneous strict separation: one is to impose a stronger singularity on the potential \cite[Remark 7.1]{MZ04} (see also \cite{LP} for some recent improvements); the other one is to introduce a viscous term $\epsilon\partial_t\varphi$ in the chemical potential $\mu$, which brings some regularizing effect on the time derivative of $\varphi$ (see \cite[Section 3]{MZ04}). In this paper, we choose to include the viscous term and extend the result in \cite{MZ04} to the coupled system \eqref{f1.a}--\eqref{f2.b} with chemotaxis, active transport and Oono's interaction. The argument in \cite{MZ04} is essentially based on a comparison principle for second-order parabolic equations which is only available for $\epsilon>0$. In our case, additional efforts have to be made to overcome the difficulties brought by the nonlocal Oono's term (which yields the loss of mass conservation) as well as the coupling with the nutrient $\sigma$ when $\chi\neq 0$. 	

Finally, we study the long-time behavior of problem \eqref{f1.a}--\eqref{ini0}. Based on some dissipative estimates and the asymptotic compactness of global weak solutions, we prove the existence of a global attractor for the corresponding dissipative dynamical system generated by global weak solutions of problem \eqref{f1.a}--\eqref{ini0} in a proper phase space (see Theorem \ref{attr}). In this aspect, we mention \cite{FGR,MRS} for the study on global attractors of some related diffusive interface models with regular potentials and neglecting the effects due to chemotaxis as well as active transport. At last, we characterize the $\omega$-limit set of an arbitrary initial datum in the finite energy space (see Theorem \ref{main4}). This also provides a dynamical approach for the investigation of steady states to problem \eqref{f1.a}--\eqref{ini0} even for the non-viscous case $\epsilon=0$ (see Corollary \ref{sta2}). We remark that there are some further issues worth investigating, such as the convergence of any bounded global weak solution to a single steady state as $t\to +\infty$ (cf. \cite{A2007,GG}) and the asymptotic behavior of solutions as $\epsilon\to 0$ (cf. \cite{MZ04}). These will be addressed in a forthcoming paper.

The remaining part of this paper is organized as follows. In Section \ref{pm}, we introduce the functional settings and state the main results. Section \ref{ws1} is devoted to the existence and uniqueness of global weak solutions. In Section \ref{ss}, we prove the existence, uniqueness and the strict separation property of global strong solutions. In Section \ref{ce}, we first show the instantaneous regularity and in particular, the separation property for global weak solutions for $t>0$. Then we prove the existence of a global attractor and characterize the $\omega$-limit set.

\section{Main Results}\label{pm}
\setcounter{equation}{0}
\subsection{Preliminaries}
We assume that $\Omega \subset\mathbb{R}^3$ is a bounded domain with smooth boundary $\partial\Omega$. For the standard Lebesgue and Sobolev spaces, we use the notations $L^{p} := L^{p}(\Omega)$ and $W^{k,p} := W^{k,p}(\Omega)$ for any $p \in [1,+\infty]$, $k \in \mathbb{N}\setminus \{0\} $ equipped with the norms $\|\cdot\|_{L^{p}}$ and $\|\cdot\|_{W^{k,p}}$.  In the case $p = 2$, we denote $H^{k} := W^{k,2}$ with the norm $\|\cdot\|_{H^{k}}$. The norm and inner product on $L^{2}(\Omega)$ are simply denoted by $\|\cdot\|$ and $(\cdot,\cdot)$, respectively.
The dual space of a Banach space $X$ is denoted by $X'$, and the duality pairing between $X$ and its dual will be denoted by
$\langle \cdot,\cdot\rangle_{X',X}$. Given an interval $J\subset \mathbb{R}^+$, we introduce
the function space $L^p(J;X)$ with $p\in [1,+\infty]$, which
consists of Bochner measurable $p$-integrable
functions with values in the Banach space $X$.

For every $f\in H^1(\Omega)'$, we denote by $\overline{f}$ its generalized mean value over $\Omega$ such that
$\overline{f}=|\Omega|^{-1}\langle f,1\rangle_{(H^1)',H^1}$. If $f\in L^1(\Omega)$, then its mean is simply given by $\overline{f} =|\Omega|^{-1}\int_\Omega f \,dx$. Besides, in view of the homogeneous Neumann boundary condition \eqref{boundary}, we also set
$H^2_{N}(\Omega):=\{f\in H^2(\Omega):\,\partial_{\bm{n}}f=0 \ \textrm{on}\  \partial \Omega\}$ and recall the following  Poincar\'{e}--Wirtinger inequality:
\begin{equation}
\label{poincare}
\|f-\overline{f}\|\leq C_P\|\nabla f\|,\quad \forall\,
f\in H^1(\Omega),
\end{equation}
where $C_P$ is a constant depending only on the dimensions and $\Omega$.
Consider the realization of the operator $-\Delta$  with homogeneous Neumann boundary condition denoted by $\mathcal{A}_N$ such that $\mathcal{A}_N\in \mathcal{L}(H^1(\Omega),H^1(\Omega)')$ is defined by
\begin{equation}\nonumber
   \langle \mathcal{A}_N u,v\rangle_{(H^1)',H^1} := \int_\Omega \nabla u\cdot \nabla v \, dx,\quad \text{for }\,u,v\in H^1(\Omega).
\end{equation}
Then for the linear subspaces
$$
V_0=\{ u \in H^1(\Omega):\ \overline{u}=0\}, \quad
V_0'= \{ u \in H^1(\Omega)':\ \overline{u}=0 \},
$$
the restriction of $\mathcal{A}_N$ from $V_0$ onto $V_0'$
is an isomorphism. In particular, $\mathcal{A}_N$ is positively defined on $V_0$ and self-adjoint. We denote its inverse map by $\mathcal{N} =\mathcal{A}_N^{-1}: V_0'
\to V_0$. Note that for every $f\in V_0'$, $u= \mathcal{N} f \in V_0$ is the unique weak solution of the Neumann problem
$$
-\Delta u=f \quad \text{in} \ \Omega,\quad \text{with}\quad
\partial_{\bm{n}} u=0 \quad \text{on}\ \partial \Omega.
$$
Besides, we have (cf. \cite{MZ04})
\begin{align}
&\langle \mathcal{A}_N u, \mathcal{N} g\rangle_{V_0',V_0} =\langle  g,u\rangle_{(H^1)',H^1}, \quad \forall\, u\in H^1(\Omega), \ \forall\, g\in V_0',\label{propN1}\\
&\langle  g, \mathcal{N} f\rangle_{V_0',V_0}
=\langle f, \mathcal{N} g\rangle_{V_0',V_0} = \int_{\Omega} \nabla(\mathcal{N} g)
\cdot \nabla (\mathcal{N} f) \, dx, \quad \forall \, g,f \in V_0',\label{propN2}
\end{align}
and the chain rule
\begin{align}
&\langle \partial_t u, \mathcal{N} u(t)\rangle_{V_0',V_0}=\frac{1}{2}\frac{d}{dt}\|\nabla \mathcal{N} u\|^2,\ \ \textrm{a.e. in}\ (0,T),\nonumber
\end{align}
for any $u\in H^1(0,T; V_0')$. For any $f\in V_0'$, we set $\|f\|_{V_0'}=\|\nabla \mathcal{N} f\|$.
It is well-known that $f \to \|f\|_{V_0'}$ and $
f \to(\|f-\overline{f}\|_{V_0'}^2 +|\overline{f}|^2)^\frac12$ are
equivalent norms on $V_0'$ and $H^1(\Omega)'$,
respectively. Besides, thanks to \eqref{poincare}, we
see that $f\to \|\nabla f\|$,  $f\to (\|\nabla f\|^2+|\overline{f}|^2)^\frac12$ are equivalent norms on $V_0$ and $H^1(\Omega)$, and
$
\|f\| \leq \|f\|_{V_0'}^{\frac12} \| \nabla f\|^{\frac12}$ for any $f \in V_0$.
We also consider the operator $\mathcal{A}_1 := I-\Delta $ with homogeneous Neumann boundary condition that is an unbounded operator $L^2(\Omega)$ with domain $D(\mathcal{A}_1) =H^2_N(\Omega)$. It is well-known that
$\mathcal{A}_1$ is a positive, unbounded, self-adjoint operator in $L^2(\Omega)$ with a compact inverse (denoted by $\mathcal{N}_1:=\mathcal{A}^{-1}_1$), see, e.g., \cite[Chapter II, Section 2.2]{T}. Then $f \to \|\mathcal{N}_1^\frac12 f\|$ is also an equivalent norm on $H^1(\Omega)'$.

In the sequel, the symbol $C$ denotes a generic positive constant that may depend on norms of the initial data, the domain $\Omega$ as well as parameters of the system. We denote by $C_T$ if a positive constant depends on the final time $T$. Specific dependence will be pointed out if necessary.

\subsection{Main results}
We now introduce the following hypotheses.
\begin{enumerate}
	\item[(H1)]\label{item:as} The singular potential $\varPsi$ belongs to the class of functions $C[-1,1]\cap C^{2}(-1,1)$ and can be written into the following form
	\begin{equation} \varPsi(r)=\varPsi_{0}(r)-\frac{\theta_{0}}{2}r^2,\nonumber
	\end{equation}
	such that
	\begin{equation}
	\lim_{r\to \pm 1} \varPsi_{0}'(r)=\pm \infty ,\quad \text{and}\ \  \varPsi_{0}''(r)\ge \theta,\quad \forall\, r\in (-1,1),\nonumber
	\end{equation}
 with some strictly positive constants $\theta_{0},\  \theta$ satisfying
\be
\theta_{0}-\theta:=K>0.
\ee
We make the extension $\varPsi_{0}(r)=+\infty$ for any $r\notin[-1,1]$.
 In addition, there exists a small constant $a_0\in(0,1)$ such that $\varPsi_{0}''$ is non-decreasing in $[1-a_0,1)$ and non-increasing in $(-1,-1+a_0]$.
	
	\item[(H2)]\label{sco} The coefficients $A,\ B,\ \epsilon,\ \chi,\ \alpha,\ c_0,$ are  prescribed  constants and satisfy
	\be
	A>0,\ \ B>0,\ \epsilon>0,\  \alpha> 0,\ \chi \in \mathbb{R},\ c_0\in(-1,1).
\nonumber
	\ee
\end{enumerate}
\begin{remark}
	It is easy to verify that the logarithmic potential \eqref{pot} fulfills the assumption $\mathrm{(H1)}$. This is indeed the case we are interested in.
\end{remark}

Next, we introduce the notions of weak and strong  solutions to the initial and boundary value problem \eqref{f1.a}--\eqref{ini0}.
\bd[Weak solutions] \label{maind}
A triple $(\varphi,\mu,\sigma)$ is a global weak solution to problem \eqref{f1.a}--\eqref{ini0}, if it fulfills the following regularity properties
\begin{align}
&\varphi \in C([0,+\infty);H^1(\Omega))\cap L^{2}_{loc}(0,+\infty;H^2_{N}(\Omega)),\notag \\
&\partial_t\varphi \in L^2(0,+\infty;L^2(\Omega)),\notag \\
&\mu \in   L^{2}_{loc}(0,+\infty;H^1(\Omega)),\notag \\
&\sigma  \in C([0,+\infty);L^2(\Omega))\cap L^{2}_{loc}(0,+\infty;H^1(\Omega)),\notag \\
&\partial_t\sigma\in  L^2(0,+\infty;H^1(\Omega)'),\notag\\
&\text{with}\ \varphi\in L^{\infty}(\Omega\times (0,+\infty))\ \textrm{and}\ \ |\varphi(x,t)|<1\ \ \textrm{a.e.\ in}\ \Omega\times(0,+\infty),\notag
\end{align}
and satisfies
\begin{subequations}
			\begin{alignat}{3}
			&\left \langle \partial_t \varphi,\xi\right \rangle_{(H^1)',H^1}=- (\nabla \mu,\nabla \xi)-\alpha(\varphi-c_0,\xi),\, \quad\quad \textrm{a.e.\ in}\ (0,+\infty), \label{test1.a} \\
			&\ \, \mu=A\varPsi'(\varphi)-B\Delta \varphi-\chi \sigma+\epsilon\partial_t\varphi,\quad\  \quad\quad\quad \textrm{a.e.\ in}\ \Omega\times(0,+\infty),\label{test4.d}\\
			&\left \langle\partial_t \sigma,\xi\right \rangle_{(H^1)',H^1}+ (\nabla \sigma,\nabla \xi)= \chi ( \nabla \varphi,\nabla \xi),\  \ \quad\quad\quad \textrm{a.e.\ in}\ (0,+\infty),\label{test2.b}
			\end{alignat}
\end{subequations}
for all $\xi \in H^1(\Omega)$. Moreover, the initial conditions are fulfilled $\varphi|_{t=0}=\varphi_{0}$, $\sigma|_{t=0}=\sigma_{0}$.
\ed
\bd[Strong solutions] \label{maind1}
 A triple $(\varphi,\mu,\sigma)$ is a global strong solution to problem \eqref{f1.a}--\eqref{ini0}, if it fulfills the following regularity properties
\be
\begin{aligned}
    &\varphi \in  C([0,+\infty);H^2_N(\Omega))\cap L^2_{loc}(0,+\infty; H^3(\Omega)), \\
	&\partial_t\varphi\in L^{\infty}(0,+\infty;L^2(\Omega))\cap L^{2}(0,+\infty;H^1(\Omega)), \\
	&\mu \in   L^{\infty}(0,+\infty;L^2(\Omega))\cap  L^{2}_{loc}(0,+\infty;H^1(\Omega)), \\
	&\sigma  \in C([0,+\infty);H^2_N(\Omega))\cap L^{2}_{loc}(0,+\infty;H^3(\Omega)),\\
&\partial_t\sigma  \in  L^{\infty}(0,+\infty;L^2(\Omega))\cap L^{2}(0,+\infty;H^1(\Omega)),
\label{sp}
\end{aligned}
\ee
and satisfies the system \eqref{f1.a}--\eqref{ini0} $\textrm{a.e.\ in}\ \Omega\times(0,+\infty)$ as well as the initial conditions.
\ed

Now we state the main results of this paper.
\bt[Global weak solutions]
\label{main1}
Suppose that the hypotheses (H1)--(H2) are satisfied, then for any initial data satisfying $\varphi_{0}\in H^1(\Omega)$, $\sigma_{0}\in L^2(\Omega)$ with $\|  \varphi_{0} \|_{L^{\infty}} \le 1$ and
$|\overline{\varphi_0}|<1$, problem \eqref{f1.a}--\eqref{ini0} admits a unique global weak solution $(\varphi,\mu,\sigma)$ in the sense of Definition \ref{maind}. Moreover, consider two groups of initial data satisfying  $(\varphi_{0i},\sigma_{0i})\in H^1(\Omega)\times L^2(\Omega)$ with $\left \|  \varphi_{0i}\right \|_{L^{\infty}} \le 1$, $|\overline{\varphi_{0i}}|<1$, $i=1,\, 2$ and $T>0$.
The global weak solutions $(\varphi_{1},\sigma_{1})$,   $(\varphi_{2 },\sigma_{2})$ to problem \eqref{f1.a}--\eqref{ini0} on $[0,T]$ with initial data $(\varphi_{0i},\sigma_{0i})$, $i=1,\, 2$,
satisfy the following continuous dependence estimate:
\be
\begin{aligned}
&\|\varphi_1(t)-\varphi_2(t)\|_{(H^1)'}^2 + \epsilon\|\varphi_1(t)-\varphi_2(t)\|^2 + \|\sigma_1(t)-\sigma_2(t)\|_{(H^1)'}^2\\
&\qquad + \int_0^t\|\varphi_1(\tau)-\varphi_2(\tau)\|_{H^1}^2 d\tau  +
\int_0^t\|\sigma_1(\tau)-\sigma_2(\tau)\|^2 d\tau \\
&\quad \le C_T\big(\|\varphi_{01}-\varphi_{02}\|_{(H^1)'}^2 + \epsilon\|\varphi_{01}-\varphi_{02}\|^2 +\|\sigma_{01}-\sigma_{02}\|_{(H^1)'}^2 +|\overline{\varphi_{01}}-\overline{\varphi_{02}}| \big),
\label{conti2}
\end{aligned}
\ee
for all $t\in [0,T]$, where the constant ${C}_T>0$ may depend on norms of the initial data, $\Omega$, $T$ and coefficients of the system.
\et
\bt[Global strong solutions]
\label{main} Suppose that the hypotheses (H1)--(H2) are satisfied, then for any initial data satisfying $\varphi_{0} \in H^{2}_N(\Omega)$, $\sigma_{0}\in H^2_N(\Omega)$ with $|\overline{\varphi_0}|<1$ and
\be
\|\varphi_0\|_{L^{\infty}}\le 1-\delta_0,
\label{sep3}
\ee
for some $\delta_0\in(0,1)$, problem \eqref{f1.a}--\eqref{ini0} admits a unique global strong solution in the sense of Definition \ref{maind1}.
Furthermore, there exists a constant $\delta\in(0, \delta_0]$ such that
\be
\|\varphi(t)\|_{L^{\infty}} \leq 1-\delta,\quad \forall\, t \in[0,+\infty),\label{sep1}
\ee
where $\delta$ depends on $\delta_0$, norms of the initial data, $\Omega$, and coefficients of the system.
\et
By virtue of the well-posedness of global
strong solutions, we are able to prove that the weak solution regularizes instantaneously for $t>0$.
\begin{corollary}
[Regularity of weak solutions] \label{cr}
Suppose that the assumptions of Theorem \ref{main1} are satisfied. For any $\tau>0$, the global weak solution $(\varphi, \sigma)$ obtained in Theorem \ref{main1} becomes a strong one on $[\tau,+\infty)$ and there exists a constant $\delta_\tau\in (0,1)$ such that
\be
\|\varphi(t)\|_{L^{\infty}}
\leq 1-\delta_\tau,\quad \forall\, t \in[\tau,+\infty),
\label{sep7}
\ee
where $\delta_\tau$ depends on $\tau$, $\|\varphi_0\|_{H^1}$, $\|\sigma_0\|$, $1-|\overline{\varphi_0}|$, $\Omega$ and coefficients of the system.
\end{corollary}

Next, we state the results on the long-time behavior.
For any given $m_1\in [0,1)$, $c_0\in [-m_1,m_1]$ and $m_2\geq 0$, we introduce the following phase space
$$
\mathcal{X}_{m_1,m_2}=\{(\varphi,\sigma)\in H^1(\Omega)\times L^2(\Omega)\,:\, \|\varphi\|_{L^\infty}\leq 1,\ |\overline{\varphi}|\leq m_1,\ |\overline{\sigma}|\leq m_2\}
$$
endowed with the metric
$
\mathrm{d}((\varphi_1,\sigma_1),(\varphi_2,\sigma_2))=
\|\varphi_1-\varphi_2\|_{H^1}+\|\sigma_1-\sigma_2\|
$.
It is straightforward to check that $\mathcal{X}_{m_1,m_2}$ is a complete metric space. Then we have
\begin{theorem}[Global attractor]\label{attr}
Suppose that the hypotheses (H1)--(H2) are satisfied. Problem \eqref{f1.a}--\eqref{ini0} generates a dynamical system $\mathcal{S}(t): \mathcal{X}_{m_1,m_2}\to \mathcal{X}_{m_1,m_2}$ such that $\mathcal{S}(t)(\varphi_0,\sigma_0) =(\varphi(t),\sigma(t))$ for all $t\geq 0$, where $(\varphi, \sigma)$ is the global weak solution in the sense of Definition \ref{maind} corresponding
to the initial datum $(\varphi_0, \sigma_0)\in \mathcal{X}_{m_1,m_2}$. Moreover, the dynamic system $(\mathcal{S}(t), \mathcal{X}_{m_1,m_2})$ possesses a compact global attractor $\mathcal{A}_{m_1,m_2}\subset \mathcal{X}_{m_1,m_2}$, which is bounded in $H^2(\Omega)\times H^2(\Omega)$.
\end{theorem}

 The last result concerns the sequential convergence of global weak solutions as time goes to infinity. More precisely,
\bt[Convergence to equilibrium] \label{main4}
Suppose that the hypotheses (H1)--(H2) are satisfied. Let $(\varphi, \sigma)$ be a global weak solution to problem \eqref{f1.a}--\eqref{ini0} with initial datum $(\varphi_0, \sigma_0)$ as given in Theorem \ref{main1}. Then there exists an equilibrium $(\varphi_{\infty},\sigma_{\infty}) \in H_{N}^{2}(\Omega)\times H_{N}^{2}(\Omega)$, that is a strong solution to the following stationary problem
\begin{subequations}
	\begin{alignat}{3}
&-B\Delta \varphi_{\infty} +A\varPsi^{\prime}(\varphi_{\infty}) -\chi\sigma_{\infty} +\alpha \mathcal{N}(\varphi_{\infty}-c_0) =A\overline{\varPsi^{\prime}(\varphi_{\infty})}
-\chi\overline{\sigma_{\infty}},\quad
&\ \text{a.e. in } \Omega,   \label{5bchv}\\
& \Delta (\sigma_{\infty}-\chi\varphi_{\infty})=0,
&\ \text{a.e. in } \Omega,   \label{5f2.b}  \\
&\partial_{\bm{n}} \varphi_{\infty} = \partial_{\bm{n}} \sigma_{\infty}=0,
&\ \text{on } \partial \Omega, \label{5cchv}\\
&\text{with}\quad \overline{\varphi_{\infty}}=c_0, \quad   \overline{\sigma_{\infty}}=\overline{\sigma_{0}},
& \label{5dchv}
	\end{alignat}
\end{subequations}
and an increasing unbounded sequence $\{t_n\}\nearrow +\infty$ such that
$$
(\varphi(t_n), \sigma(t_n)) \rightarrow (\varphi_{\infty},\sigma_\infty) \quad  \text { strongly in } H^{2r}(\Omega)\times H^{2r}(\Omega),
$$
for some $r\in (1/2,1)$.
Furthermore, there exists a constant $\delta_{\infty}\in (0,1)$ such that
\be
\|\varphi_{\infty}\|_{L^{\infty}} \leq 1-\delta_{\infty}.
\label{sep10}
\ee
\et

\begin{remark}
$\mathrm{(1)}$ Similar results can be obtained when the spatial dimension is two. Indeed, in the two dimensional case, one can prove the instantaneous separation property for the phase function $\varphi$ when the viscous term vanishes, i.e., $\epsilon=0$, by extending the arguments in \cite{MZ04,GG}. $\mathrm{(2)}$ In $\mathrm{(H2)}$, we have assumed that $\alpha>0$, i.e., including the Oono's interaction in the analysis. Nevertheless, with minor modifications, the case $\alpha=0$ can be easily treated, keeping in mind that the mass is now conserved.
\end{remark}

\section{Global Weak Solutions}\label{ws1}
\setcounter{equation}{0}
In this section, we prove Theorem \ref{main1} on the existence and uniqueness of global weak solutions to problem \eqref{f1.a}--\eqref{ini0}.

\subsection{Existence}\label{sgs}

For the singular potential $\varPsi$ satisfying (H1), without loss of generality, we assume that  $\varPsi_0(0)=0$. Then we can approximate the singular part $\varPsi_0'$, e.g., as in \cite{MT}:
\begin{equation}
\varPsi'_ {0,\kappa}(r)=\left\{
\begin{aligned}
&\varPsi_0'(-1+\kappa) + \varPsi_0''(-1+\kappa)(r+1-\kappa),\quad r<-1+\kappa,\\
&\varPsi_0'(r),\qquad\qquad\qquad\qquad\qquad\qquad\qquad\ |r|\leq 1-\kappa,\\
&\varPsi_0'(1-\kappa) + \varPsi_0''(1-\kappa)(r-1+\kappa),\qquad\ \  r>1-\kappa,
\end{aligned}
\right.\label{vPsi}
\end{equation}
for a sufficiently small $\kappa\in(0,a_0)$.
Define
$$
\varPsi_{0,\kappa}(r)=\int_0^{r}\varPsi'_{0,\kappa}(s) \, ds,\quad \varPsi_\kappa(r)=\varPsi_{0,\kappa}(r)-\frac{\theta_0}{2}r^2.
$$
We can verify that $\varPsi_ {0,\kappa}''(r)\ge \theta>0$ and $\varPsi_ {0,\kappa}(r)\geq -L$ for $r\in \mathbb{R}$, where $L>0$ is a constant independent of $\kappa$. Moreover, it holds $\varPsi_{0,\kappa}(r)\leq \varPsi_{0}(r)$ for $r\in [-1,1]$ (see, e.g., \cite{GGW}).

Let $T>0$ be given. We consider the following approximate problem: looking for functions
and $(\varphi^\kappa,\mu^{\kappa},\sigma^\kappa)$ satisfying
\begin{align}
&\left \langle\partial_t \varphi^\kappa,\xi\right \rangle_{(H^1)',{H}^1} =-(\nabla \mu^{\kappa},\nabla\xi)-\alpha(\varphi^\kappa-c_0,\xi),
\label{pag4.d} \\
&\mu^{\kappa}=A\varPsi'(\varphi^{\kappa})-B\Delta \varphi^{\kappa}-\chi \sigma^{\kappa}+\epsilon\partial_t \varphi^\kappa, \label{pag1.a}\\
&\left \langle\partial_t \sigma^\kappa,\xi\right \rangle_{(H^1)',{H}^1}+(\nabla \sigma^\kappa,\nabla \xi) = \chi ( \nabla \varphi^\kappa,\nabla \xi),\qquad\,   \label{pag2.b}\\
&\varphi^{\kappa}|_{t=0}=\varphi_{0},\quad  \sigma^{\kappa}|_{t=0}=\sigma_{0},\label{pag5.e}
\end{align}
in $\Omega\times(0,T)$ for all $\xi \in H^1(\Omega)$.

For every $\kappa\in(0,a_0)$, the regularized problem \eqref{pag4.d}--\eqref{pag5.e} can be solved via a standard Galerkin scheme like in \cite{GL17e}. Then following the arguments in \cite[Appendix]{H}, one can derive uniform estimates of the approximate solutions $(\varphi^\kappa,\mu^{\kappa},\sigma^\kappa)$ with respect to the parameter $\kappa$ and time on a given interval $[0,T]$.
In particular, the additional viscous term $\epsilon\partial_t\varphi^\kappa$ in \eqref{pag1.a} yields some additional regularity on $\partial_t\varphi^\kappa$ such that $\epsilon\partial_t\varphi^\kappa\in L^2(0,T;L^2(\Omega))$. Then one can employ a compactness argument (see e.g., \cite{MT}) to pass to the limit as $\kappa\to 0$ to obtain a global weak solution $(\varphi, \mu, \sigma)$ of the original problem \eqref{f1.a}--\eqref{ini0} on an arbitrary interval $[0,T]$. Since the procedure is standard, we omit the details here.

\subsection{Uniform-in-time estimates}\label{ws}

In view of the above discussion, here we confine ourselves to derive sufficient uniform-in-time estimates that allow us to extend the weak solution $(\varphi, \mu, \sigma)$ to the whole interval $[0,+\infty)$, and thus complete the proof for existence part of Theorem \ref{main1}. We also note that the dissipative estimate obtained below will enable us to study the long-time behavior of problem \eqref{f1.a}--\eqref{ini0} in Section \ref{ce}. \medskip

 \textbf{First estimate}.  Testing the equation \eqref{test1.a} by $1$, we obtain
\be
\frac{d}{dt}\big(\overline{\varphi}(t)-c_0\big) +\alpha\big(\overline{\varphi}(t)-c_0\big)=0,
\label{conver1}
\ee
so that
\be
\overline{\varphi}(t)=c_0+e^{-\alpha t}\left(\overline{\varphi_0}-c_0\right),\quad\forall\,t\geq 0.\label{aver}
\ee
Similarly, choosing the test function $\xi = 1$ in \eqref{test2.b}, we have
\be
\overline{\sigma}(t)=\overline{\sigma_0}, \quad\forall\,t\geq 0.\label{1aver}
\ee

\textbf{Second estimate}. As it has been shown in \cite{GG}, the convex part of the free energy
$$
\mathcal{E}^*(\varphi)=\frac{B}{2}\|\nabla \varphi\|^2 + A\int_\Omega \varPsi_0(\varphi) dx, \quad \forall\, \varphi\in H^1(\Omega),
$$
is a proper, lower semi-continuous
and convex functional. Then for $(\varphi, \mu, \sigma)$, a weak solution in the sense of Definition \ref{maind}, it follows from the standard chain rule that
$$
\frac{d}{dt}\mathcal{E}^*(\varphi)=(\partial_t\varphi, \mu+\theta_0\varphi+\chi\sigma -\epsilon\partial_t\varphi).
$$
Hence, we can obtain the energy identity for weak solutions by testing \eqref{test4.d} with $\mu$, \eqref{test1.a} with $\partial_t \varphi$, \eqref{test2.b} with $\sigma-\chi\varphi$, and adding the resultants together,
\begin{align}
&\frac{d}{dt} \mathcal{E}(t)  +\mathcal{D}(t)+\int_\Omega\alpha(\varphi(t)-c_0)\mu(t)\, dx=0,\quad \forall\,t\geq 0,
\label{1menergy}
\end{align}
where
\begin{align*}
\mathcal{E}(t)&= \int_\Omega \Big( A\varPsi(\varphi(t))- \chi\sigma(t)\varphi(t)\Big)\, dx + \frac{B}{2}\|\nabla \varphi(t)\|^2
+\frac{1}{2}\|\sigma(t)\|^2,\notag\\
\mathcal{D}(t)& =\|\nabla \mu(t)\|^2 +\|\nabla(\sigma(t)-\chi\varphi(t))\|^2 +\epsilon\|\partial_t\varphi(t)\|^2.
\end{align*}
Since
\be
\|\varphi(t)\|_{L^{\infty}} \leq 1,\quad \forall\,t\geq 0,\label{phi}
\ee
we deduce from the Poincar\'{e}--Wirtinger inequality \eqref{poincare} and \eqref{aver}, \eqref{1aver} that
\begin{align}
\|\sigma\|&\le\|\sigma-\chi \varphi\|+\|\chi\varphi\|\notag\\
&\le C_P\|\nabla(\sigma-\chi\varphi)\| +|\Omega||\overline{\sigma -\chi\varphi}|+\|\chi \varphi\|\notag\\
&\le C_P\|\nabla(\sigma-\chi\varphi)\|+C_1,
\label{sig1}
\end{align}
where the constant $C_1>0$ only depends on $\chi$, $\Omega$, $|\overline{\sigma_0}|$.
We recall the following inequality due to the hypotheses (H1) (see, e.g., \cite{GG})
\be
\varPsi(r)\leq \varPsi(s) + \varPsi(r)(r-s) + \frac{K}{2}(r-s)^2\quad\forall\, r,\, s  \in(-1,1).\label{ine}
\ee
Then by \eqref{phi}--\eqref{ine} and Young's inequality, we can follow the argument in \cite{GG,H} to obtain that
\begin{align}
&\alpha\int_\Omega(\varphi-c_0)\mu\,dx\notag\\
&\quad = \alpha B\|\nabla \varphi\|^2+\alpha A\int_\Omega (\varphi-c_0)\varPsi'(\varphi)\,dx -\alpha\chi \int_\Omega (\varphi-c_0)\sigma\,dx \notag\\
&\qquad
+\alpha\epsilon\int_\Omega (\varphi-c_0)\partial_t\varphi\,dx
\notag\\
&\quad
\geq {\alpha B}\|\nabla \varphi\|^2
+\alpha A\int_\Omega \Big(\varPsi(\varphi)-\varPsi(c_0) -\frac{K}{2}(\varphi-c_0)^2\Big)\,dx
\notag\\
&\qquad
-\alpha\chi \int_\Omega (\varphi-c_0)\sigma\,dx +\alpha\epsilon\int_\Omega (\varphi-c_0)\partial_t\varphi\,dx
\notag\\
&\quad
\geq \min\Big\{\alpha,\frac{1}{4C_P^2}\Big\} \left(\mathcal{E}+ \chi \int_\Omega \sigma\varphi\, dx -\frac{1}{2}\|\sigma\|^2\right)
-\frac{1}{4C_P^2}\| \sigma\|^2
\notag \\
&\qquad
-\Big(\frac{\alpha AK}{2} +C_P^2\alpha^2\chi^2 +\frac{\alpha^2\epsilon}{2}\Big) \|\varphi-c_0\|^2 -\frac{\epsilon}{2}\|\partial_t\varphi\|^2
-A\alpha|\varPsi(c_0)||\Omega|
\notag\\
&\quad
\geq \min\Big\{\alpha,\frac{1}{4C_P^2}\Big\} \mathcal{E}
-\frac{\chi^2}{8C_P^2}\| \varphi\|^2 -\frac{1}{2C_P^2}\| \sigma\|^2
\notag\\
&\qquad
-\Big(\frac{\alpha AK}{2}+C_P^2\alpha^2\chi^2 +\frac{\alpha^2\epsilon}{2}\Big) \|\varphi-c_0\|^2 -\frac{\epsilon}{2}\|\partial_t\varphi\|^2
-A\alpha|\varPsi(c_0)\|\Omega|
\notag\\
&\quad \geq
\min\Big\{\alpha,\frac{1}{4C_P^2}\Big\}\mathcal{E}
-\frac{1}{2}\| \nabla(\sigma-\chi\varphi)\|^2 -\frac{\epsilon}{2}\|\partial_t\varphi\|^2
-C_2 \notag\\
&\quad
\geq \min\Big\{\alpha,\frac{1}{4C_P^2}\Big\} \mathcal{E} - \frac{1}{2}\mathcal{D} -C_2,
\label{mueeb}
\end{align}
where the constant $C_2>0$ only depends on $\alpha$, $A$, $B$, $c_0$, $\theta$, $\theta_0$, $\chi$, $\Omega$ and $\overline{\sigma_0}$.

Taking $c_{*}:=\min\{\alpha,1/(4C_P^2)\}$, we deduce from \eqref{1menergy} and \eqref{mueeb} that
\begin{align}
&\frac{d}{dt} \mathcal{E}(t)  + c_*\mathcal{E}(t) + \frac12 \mathcal{D}(t) \le C_2,\quad \forall\,t\geq 0.
\label{2menergy}
\end{align}
By Gronwall's lemma, we deduce from \eqref{2menergy} the following dissipative estimates
\begin{align}
& \mathcal{E}(t)  \leq \mathcal{E}(0)e^{-c_* t} + \frac{C_2}{c_*},\quad \forall\, t\geq 0,
\label{diss1}\\
& \int_{t}^{t+1} \mathcal{D}(\tau) d\tau \leq
2\mathcal{E}(0)e^{-c_* t} +2C_2(1+c_*^{-1}),\quad \forall\,t\geq 0.
\label{diss2}
\end{align}
As a consequence, we obtain from \eqref{diss1}, \eqref{diss2} and \eqref{phi} that
\begin{align}
&\left \| \varphi\right \|_{L^{\infty}(0,+\infty;H^1(\Omega))}
+\left \| \sigma\right \|_{L^{\infty}(0,+\infty;L^2(\Omega))}\le C,\label{prioria20} \\
& \| \sigma  \|_{L^{2}(0,T;H^1(\Omega))}
+ \| \partial_t\varphi\|_{L^2(0,T;L^2(\Omega))}
+ \| \nabla  \mu \|_{L^{2}(0,T;L^2(\Omega))} \le C_T,
\label{prioria4}
\end{align}
where the constant $C>0$ is independent of time, and  $C_T>0$ may depend on $T$ (with $T>0$ being an arbitrary but fixed final time).
In particular, the uniform-in-time estimates \eqref{prioria20} enable us to extend the weak solutions $(\varphi,\mu,\sigma)$ from an arbitrary interval $[0,T]$ to the whole interval $[0,+\infty)$, being uniformly bounded in the corresponding spaces.

\medskip
\textbf{Third estimate}.
Testing \eqref{test4.d} by $\xi=1$, using integration by parts and the relations \eqref{aver}, \eqref{1aver}, we get
\begin{align}
|\overline{\mu}|
&=|\Omega|^{-1}|A(\varPsi'(\varphi),1)+\epsilon(\partial_t \varphi,1)| \notag \\
& \le |\Omega|^{-1}A\| \varPsi'(\varphi)\|_{L^1}  +|\Omega|^{-1}\epsilon \alpha (|\overline{\varphi_0}|+|c_0|).
\label{average valuea}
\end{align}
The term $\| \varPsi'(\varphi)\|_{L^1}$ can be estimated as in \cite[Section 3]{GG} (see also  \cite{MZ04}) with minor modifications such that
\begin{equation}
\begin{aligned}
&\| \varPsi'(\varphi)\|_{L^1}
\le C(\|  \nabla \mu\| + \|\nabla \varphi\|+ \|\sigma\|+ \|  \partial_t \varphi\|)\|\nabla \varphi\|\leq C(\|  \nabla \mu\| + \|  \partial_t \varphi\|+1),
\label{varPsi}
\end{aligned}
\end{equation}
where the constant $C>0$ depends on the initial energy $\mathcal{E}(0)$ and coefficients of the system.
As a consequence, from the estimates \eqref{prioria20},
\eqref{prioria4}, \eqref{average valuea} and inequality \eqref{poincare},
 we obtain
\begin{equation}
\left \|  \mu  \right \|_{L^{2}(0,T;H^1(\Omega)) }\le C_T.
\label{mu0}
\end{equation}

\textbf{Fourth estimate}. Multiplying \eqref{test4.d} with $-\Delta\varphi$ and integrating over $\Omega$, we have (see, e.g. \cite[Section 3.2.1]{H}).
\begin{align}
A(\varPsi''(\varphi)\nabla\varphi,\nabla\varphi)+B\left \|  \Delta\varphi  \right \|^2
& \le \| \nabla \mu \| \|\nabla \varphi\|
+ |\chi|\| \nabla \sigma \| \|\nabla \varphi\|+\epsilon(\partial_t\varphi,\Delta\varphi)\notag\\
&\le C(\left \| \nabla \mu \right \|+\left \| \nabla \sigma  \right\| +\|\partial_t\varphi\|^2)+\frac{B}{2}\left \|  \Delta\varphi  \right \|^2 .
\label{varP21}
\end{align}
Thus, it follows from (H1) and \eqref{prioria4}, \eqref{varP21} that
$$B\left \|  \Delta\varphi  \right \|^2
 \le C(\left \| \nabla \mu \right \|+\left \| \nabla \sigma  \right \|+ \left \|  \nabla\varphi  \right \|^2+\|\partial_t\varphi\|^2),$$
which implies
\begin{equation}
\int_0^T\left \| \Delta \varphi(t)\right \|^2 \, dt\le C_T.\notag
\end{equation}
Then by the standard elliptic estimates for the Neumann problem, we obtain
\begin{equation}
\|  \varphi\|_{L^2(0,T;H^2(\Omega))}\le C_T,
\label{phiH2}
\end{equation}
where the constant $C_T$ depends on $\mathcal{E}(0)$, $\Omega$, and coefficients of the system. Then by comparison in \eqref{f4.d} and using \eqref{prioria20}, \eqref{prioria4}, \eqref{mu0} and \eqref{phiH2}, we get
\begin{align}
\|\varPsi'(\varphi)\|_{L^2(0,T;L^2(\Omega))}\leq C_T.
\label{PsiL2}
\end{align}

\textbf{Fifth estimate}. In order to obtain an uniform-in-time estimate for $\partial_t\varphi$, we rewrite system \eqref{test1.a}--\eqref{test2.b} in the following equivalent form:
\begin{align}
&\left \langle\partial_t \varphi,\xi\right \rangle_{(H^1)',{H}^1} = -(\nabla \tilde{\mu},\nabla\xi) -\alpha(\overline{\varphi}-c_0,\xi),\label{1pag4.d} \\
&\tilde{\mu}
=A\varPsi'(\varphi)-B\Delta \varphi-\chi \sigma +\epsilon\partial_t \varphi +\alpha \mathcal{N}(\varphi-\overline{\varphi}), \label{1pag1.a}\\
&\left \langle\partial_t \sigma,\xi\right \rangle_{(H^1)',{H}^1} +(\nabla \sigma,\nabla \xi) = \chi ( \nabla \varphi,\nabla \xi),\qquad\,   \label{1pag2.b}\\
&\varphi|_{t=0}=\varphi_{0},\quad  \sigma|_{t=0}=\sigma_{0},\label{1pag5.e}
\end{align}
in $\Omega\times(0,+\infty)$ for all $\xi \in H^1(\Omega)$. Then testing \eqref{1pag4.d} with $\tilde{\mu}$, \eqref{1pag1.a} with $\partial_t \varphi$, \eqref{1pag2.b} with $\sigma-\chi\varphi$, adding the resultants together, we obtain
\begin{align}
& \frac{d}{dt}\mathcal{F}(t)  +  \mathcal{G}(t)  = - \alpha(\overline{\varphi}(t)-c_0,\tilde{\mu}(t)),\quad \forall\,t\in[0,+\infty),
\label{menergyd}
\end{align}
where
\begin{align*}
\mathcal{F}(t)
&=\int_\Omega \Big(A\varPsi(\varphi(t))-\chi\sigma(t)\varphi(t)\Big)\, dx + \frac{B}{2}\|\nabla \varphi(t)\|^2
+\frac{1}{2}\|\sigma(t)\|^2 \notag\\
&\quad
+\frac{\alpha}{2} \|\nabla\mathcal{N} (\varphi(t)-\overline{\varphi}(t))\|^{2},\notag\\
\mathcal{G}(t)& = \|\nabla \tilde\mu(t)\|^2 +\|\nabla(\sigma(t)-\chi\varphi(t))\|^2 +\epsilon\|\partial_t\varphi(t)\|^2.
\end{align*}
Similar to \eqref{average valuea} and \eqref{varPsi}, we have
\begin{equation}
\| \varPsi'(\varphi)\|_{L^1}
\le C (1+\|  \nabla \tilde{\mu}\|+\|\partial_t \varphi\|),
\label{vaps}
\end{equation}
and thus
\begin{align}
|\overline{\tilde{\mu}}|
&\le C (1+\|  \nabla \tilde{\mu}\|+\|\partial_t \varphi\|).
\label{average value1}
\end{align}
By Young's inequality, we deduce from \eqref{aver} and \eqref{average value1} that
\be
\begin{aligned}
-\alpha(\overline{\varphi}-c_0,\tilde{\mu})&\le \alpha e^{-\alpha t}|\overline{\varphi_0}-c_0||\overline{\tilde{\mu}}|\\
&\le Ce^{-\alpha t}+ \frac12\|\nabla \tilde\mu(t)\|^2 +\frac{\epsilon}{2}\|\partial_t\varphi(t)\|^2.\notag
\end{aligned}
\ee
Inserting the above estimate into \eqref{menergyd}, we deduce that
\begin{align}
& \frac{d}{dt}\mathcal{F}(t)  + \frac12 \mathcal{G}(t)  \le  Ce^{-\alpha t},\quad \forall\,t\in[0,+\infty),
\label{dmenergy1}
\end{align}
which implies
\be
\begin{aligned}
	\frac{d}{d t} \Big(\mathcal{F}(t)+\frac{C}{\alpha} e^{-\alpha t}\Big) + \frac12 \mathcal{G}(t)\leq 0,\quad \forall\,t\in[0,+\infty),\label{i6}
\end{aligned}
\ee
such that $\mathcal{F}(t)+\frac{C}{\alpha} e^{-\alpha t}$ serves as a Lyapunov functional for  problem \eqref{f1.a}--\eqref{ini0}.

Hence, integrating \eqref{i6} with respect to time, we get
\begin{align}
& \mathcal{F}(t)  + \frac12\int_0^t \mathcal{G}(\tau)\,d\tau  \le \mathcal{F}(0)+C\alpha^{-1},\quad \forall\,t\in[0,+\infty).
\label{menergy1}
\end{align}
As a consequence,
\be\left \|\nabla  \tilde{\mu}\right \|_{L^{2}(0,+\infty;L^2(\Omega))} +\epsilon\|\partial_t\varphi\|_{L^{2}(0,+\infty;L^2(\Omega))}\le C,\label{phit}
\ee
and
\be
\|\nabla(\sigma-\chi\varphi)\|_{L^{2}(0,+\infty;L^2(\Omega))}\le C,\label{phit2}
\ee
where the constant $C$ may depend on $\alpha$, $A$, $B$, $\varPsi_0$, $\theta_0$, $\chi$, $\Omega$, $\|\varphi_0\|_{H^1}$ and $\|\sigma_0\|$.
In view of \eqref{test2.b} and \eqref{phit2}, we also obtain
\be
\|\partial_t\sigma\|^2_{L^{2}(0,+\infty; H^1(\Omega)')}\le C.\label{phit1}
\ee

\subsection{Continuous dependence}\label{uw}
We now prove the continuous dependence estimate that yields the uniqueness of global weak solutions. To this end, let $(\varphi_{1},\mu_1,\sigma_{1})$ and $(\varphi_{2},\mu_2,\sigma_{2})$ be two weak solutions  of  problem \eqref{f1.a}--\eqref{ini0} given by Theorem \ref{main} corresponding to the initial data $(\varphi_{01},\sigma_{01})$ and $(\varphi_{02}, \sigma_{02})$, respectively.
Denote their differences by
\begin{align*}
& (\varphi,\mu,\,\sigma)=(\varphi_{1} -\varphi_{2},\,\mu_1-\mu_2,\, \sigma_{1}-\sigma_{2}),\\
& (\varphi_0,\sigma_0)=(\varphi_{01}-\varphi_{02},\, \sigma_{01}-\sigma_{02}).
\end{align*}
Then we have
\begin{subequations}
	\begin{alignat}{3}
	&\left \langle \partial_t \varphi,\xi\right \rangle_{(H^1)',H^1}=- (\nabla \mu,\nabla \xi)-\alpha(\varphi, \xi),\label{test11.a} \\
	&\  \mu=A\varPsi'(\varphi_{1})-A\varPsi'(\varphi_{2})-B\Delta \varphi-\chi \sigma+\epsilon\partial_t \varphi ,\label{test44.d}\\
	&\left \langle\partial_t \sigma,\xi\right \rangle_{(H^1)',H^1} + (\nabla \sigma, \nabla \xi)=\chi(\nabla \varphi,\nabla \xi), \label{test22.b}
	\end{alignat}
\end{subequations}
for almost every $t\in[0,+\infty)$ and any $ \xi \in H^1(\Omega)$.
From \eqref{test11.a}, we infer that
\begin{align}
\frac{d}{dt}\overline{\varphi}
+\alpha\overline{\varphi}=0\quad\text{as well as}\quad \frac{1}{2}\frac{d}{dt}\overline{\varphi}^2
+\alpha\overline{\varphi}^2=0.
\label{diffmean0}
\end{align}
Next, taking the test function $\xi=\mathcal{N}(\varphi-\overline{\varphi})$ in \eqref{test11.a}, we obtain
\begin{equation}
\frac{1}{2}\frac{d}{dt}\|\varphi-\overline{\varphi}\|_{V_0'}^2 +(\mu,\varphi-\overline{\varphi}) +\alpha\|\varphi-\overline{\varphi}\|_{V_0'}^2=0.
\label{diffvmb}
\end{equation}
Following the arguments in \cite{GG} (see also \cite[Section 4]{H}), we can deduce that
\begin{align}
(\mu,\varphi-\overline{\varphi})
&\ge B\|\nabla \varphi\|^2-(A|\theta_0-\theta|+\chi^2)\|\varphi\|^2 -\frac{1}{2}\|\sigma\|^2 -\chi^2|\Omega|\overline{\varphi}^2 \notag\\
&\quad -A(\varPsi'(\varphi_{1})-\varPsi'(\varphi_{2}) ,\overline{\varphi}) +\frac{\epsilon}{2}\frac{d}{dt}\|\varphi-\overline{\varphi}\|^2,\notag
\end{align}
with
\begin{align}
&(A|\theta_0-\theta|+\chi^2)\|\varphi\|^2 \le\frac{B}{2}\|\nabla \varphi\|^2+C\|\varphi\|_{(H^1)'}^2.\notag
\end{align}
Thus, from \eqref{diffmean0}, \eqref{diffvmb} and using the equivalent norm on $(H^1)'$, we get
\begin{align}
& \frac{1}{2}\frac{d}{dt}(\|\varphi\|_{(H^1)'}^2 +\epsilon\|\varphi-\overline{\varphi}\|^2)
+\frac{B}{2}\|\nabla \varphi\|^2 +\alpha|\overline{\varphi}|^2\notag\\
&\le \frac{1}{2}\|\sigma\|^2 +C\|\varphi\|_{(H^1)'}^2
+A\big(\|\varPsi'(\varphi_{1})\|_{L^1} +\|\varPsi'(\varphi_{2})\|_{L^1}\big) |\overline{\varphi}|.
\label{starphi}
\end{align}
On the other hand, taking the test function $\xi=\mathcal{N}_1\sigma$ in \eqref{test22.b}, we get
\begin{align}
&\frac12\frac{d}{dt}\|\sigma\|_{(H^1)'}^2
+ \|\sigma\|^2=(\sigma, \mathcal{N}_1\sigma)+\chi(\nabla \varphi,\nabla \mathcal{N}_1\sigma),\label{uniquedifs}
\end{align}
where
\begin{align}
(\sigma, \mathcal{N}_1\sigma)+\chi(\nabla \varphi,\nabla \mathcal{N}_1\sigma)&\le \|\sigma\|_{(H^1)'}^2+|\chi|\|\nabla \varphi\|\|\nabla \mathcal{N}_1\sigma\|\notag\\
& \le \frac{B}{4}\|\nabla\varphi\|^2
+ C\|\sigma\|_{(H^1)'}^2.\label{starphi1}
\end{align}
Collecting the above estimates, we deduce from \eqref{starphi}--\eqref{starphi1} that
\begin{align}
&\frac12 \frac{d}{dt}(\|\varphi\|_{(H^1)'}^2 +\epsilon\|\varphi-\overline{\varphi}\|^2 +\|\sigma\|_{(H^1)'}^2) + \frac{B}{4}\|\nabla\varphi\|^2 + \frac12\|\sigma\|^2\notag\\
&\quad\leq C(\|\varphi\|_{(H^1)'}^2 +\|\sigma\|_{(H^1)'}^2) +A\big(\|\varPsi'(\varphi_{1})\|_{L^1} +\|\varPsi'(\varphi_{2})\|_{L^1}\big) |\overline{\varphi}|.
\label{uniA}
\end{align}
From the first equality in \eqref{diffmean0}, we have
\be
\overline{\varphi}(t)=e^{-\alpha t}\overline{\varphi_0},\quad \forall\,t\geq 0,
\label{aver1a}
\ee
and thanks to \eqref{prioria4}, \eqref{varPsi}, it follows that
\be
\int_0^T\big(\|\varPsi'(\varphi_{1}(\tau))\|_{L^1} +\|\varPsi'(\varphi_{2}(\tau))\|_{L^1}\big) d\tau\le C_T,\quad \forall\, T \in(0,+\infty).\notag
\ee
Therefore, by an application of Gronwall's lemma to \eqref{uniA}, we obtain
\be
\begin{aligned}
&\|\varphi(t)\|_{(H^1)'}^2 +\epsilon\|(\varphi-\overline{\varphi})(t)\|^2 +\|\sigma(t)\|_{(H^1)'}^2
+\int_0^t\|\varphi(\tau)\|_{H^1}^2 d\tau  +
\int_0^t\|\sigma(\tau)\|^2 d\tau \\
&\quad\le {C}_T\big(\|\varphi_0\|_{(H^1)'}^2 +\epsilon\|\varphi_0-\overline{\varphi_0}\|^2 +\|\sigma_0\|_{(H^1)'}^2+|\overline{\varphi_0}|\big),
\quad \forall\, t\in [0,T],
\end{aligned}
\label{difes}
\ee
where the constant ${C}_T>0$ depends on norms of the initial data, $\Omega$, the coefficients of the system and $T$. Combining \eqref{difes} with \eqref{aver1a}, we arrive the conclusion of Theorem \ref{main1}.
$\hfill\square$

\section{Global Strong Solutions}\label{ss}
\setcounter{equation}{0}
In this section, we prove Theorem \ref{main} on the existence and uniqueness of global strong solutions to problem \eqref{f1.a}--\eqref{ini0}. Since a strong solution is indeed a weak solution, then its uniqueness is a direct consequence of the continuous dependence estimate \eqref{difes}. Next, concerning the existence part, recalling the arguments in \cite[Corollary 4.1]{MZ04} for the viscous Cahn--Hilliard equation, we only need to derive sufficient \textit{a priori} estimates. Different from \cite{MZ04}, here we have to overcome several additional difficulties due to the Oono's term (i.e., loss of mass conservation) and the coupling with nutrient equation.

Hence, in what follows, we provide a formal derivation of \textit{a priori} estimates for $(\varphi(t),\sigma(t))$, assuming that they are sufficiently regular functions satisfying the additional assumption
\begin{align}
\|\varphi(t)\|_{L^\infty}< 1,\quad \forall\,t\geq 0.\label{linfi}
\end{align}
It is obvious that the lower-order estimates that we have obtained in Section \ref{ws} also hold for $(\varphi,\sigma)$.\smallskip

\textbf{First estimate}. Similar to \cite[Section 3]{MZ04}, we can rewrite the problem \eqref{f1.a}--\eqref{f4.d} in the following equivalent form by using \eqref{conver1}:
\be
\begin{aligned}
& \epsilon(\partial_t\varphi -\overline{\partial_{t}{\varphi}}) +\mathcal{N}(\partial_{t}\varphi -\overline{\partial_{t}{\varphi}}) + \alpha\mathcal{N}(\varphi-\overline{\varphi})\\
&\quad =B\Delta \varphi-A\varPsi^{\prime}(\varphi) +\chi\sigma +A\overline{\varPsi^{\prime}(\varphi)} -\chi\overline{\sigma}.
\label{reeq}
\end{aligned}
\ee
Differentiating \eqref{reeq} with respect to $t$, we get
\be
\begin{aligned}
&\epsilon\partial_t(\partial_t \varphi-\overline{\partial_{t}\varphi} )
+ \partial_t\mathcal{N}(\partial_{t}\varphi -\overline{\partial_{t}\varphi}) +\alpha\mathcal{N}(\partial_{t}\varphi -\overline{\partial_{t}\varphi})\\
&\quad=B\Delta \partial_{t}\varphi -A\varPsi^{\prime\prime}(\varphi)\partial_{t}\varphi
+\chi\partial_t\sigma +A\overline{\varPsi^{\prime\prime} (\varphi)\partial_{t}\varphi},
\end{aligned}\notag
\ee
where we have used the fact \eqref{1aver}.
Multiplying the above equation by $\partial_{t}\varphi-\overline{\partial_{t}\varphi}$,  and integrating over $\Omega$, we obtain
\begin{align}
& \frac{d}{dt}\mathcal{H}(t) + B\|\nabla \partial_{t}\varphi(t)\|^{2} +  \alpha\|\nabla\mathcal{N}(\partial_{t}\varphi(t) -\overline{\partial_{t}\varphi}(t))\|^{2}
=\mathcal{J}_1(t)+\mathcal{J}_2(t),\quad \forall\,t\geq 0,
\label{menergyd2}
\end{align}
where
\begin{align*}
\mathcal{H}(t)& =\frac{\epsilon}{2}\|\partial_{t}\varphi(t) -\overline{\partial_{t}\varphi}(t)\|^{2} +\frac12\|\nabla\mathcal{N}(\partial_{t}\varphi(t) -\overline{\partial_{t}\varphi}(t))\|^{2},\notag\\
\mathcal{J}_1(t)& =(\chi\partial_t\sigma(t),\partial_{t}\varphi(t) -\overline{\partial_{t}\varphi}(t)),\notag\\
\mathcal{J}_2(t)&=(-A\varPsi^{\prime\prime}(\varphi(t)) \partial_{t}\varphi(t) +A\overline{\varPsi^{\prime\prime}(\varphi)\partial_{t}\varphi}(t), \partial_{t}\varphi(t)-\overline{\partial_{t}\varphi}(t)).
\end{align*}
The two terms on the right-hand side of \eqref{menergyd2} can be estimated as follows:
\be
\begin{aligned}
\mathcal{J}_1
&=-(\chi\Delta\mathcal{N}\partial_t\sigma, \partial_{t}\varphi-\overline{\partial_{t}\varphi})
= \chi (\nabla\mathcal{N}\partial_t\sigma,\nabla\partial_{t}\varphi)\\
	&\le \frac{B}{2}\|\nabla \partial_{t}\varphi\|^2 +\frac{\chi^2}{2B}\|\nabla\mathcal{N}\partial_t\sigma\|^2,
\end{aligned}
\notag
\ee
while for $\mathcal{J}_2$, we infer from the fact $\varPsi^{\prime\prime}(\varphi)\ge-K$ (thanks to (H1), \eqref{linfi}) and \eqref{aver} that
\be
\begin{aligned}
	&(-A\varPsi^{\prime\prime}(\varphi)\partial_{t} \varphi + A\overline{\varPsi^{\prime\prime} (\varphi)\partial_{t}\varphi}, \partial_{t}\varphi-\overline{\partial_{t}\varphi})\\
&\quad\le
AK\|\partial_{t}\varphi\|^{2} +A|\Omega|\,\overline{\varPsi^{\prime\prime}(\varphi) \partial_{t}\varphi} \,\overline{\partial_{t}\varphi}\\
&\quad =
AK\,\|\partial_{t}\varphi\|^{2} + A|\Omega|\,\frac{d}{dt}\Big(\overline{\varPsi^{\prime}(\varphi)}\, \overline{\partial_{t}\varphi}\Big)  -A|\Omega|\,\overline{\varPsi^{\prime} (\varphi)} \, \frac{d}{dt}\overline{\partial_{t}\varphi}\\
&\quad\le AK \|\partial_{t}\varphi\|^{2} +A|\Omega| \frac{d}{dt} \Big(\overline{\varPsi^{\prime}(\varphi)}\, \overline{\partial_{t}\varphi}\Big)
+  A\alpha^2 e^{-\alpha t}  \|\varPsi^{\prime}(\varphi)\|_{L^{1}} |\overline{\varphi_{0}} -c_0|.
\notag
\end{aligned}
\ee
Then it follows from \eqref{menergyd2}  and the above estimates that
\be
\begin{aligned}
&\frac{d}{dt}\mathcal{H}(t)  + \frac{B}{2} \|\nabla \partial_{t}\varphi(t)\|^{2} +  \alpha\|\nabla\mathcal{N}(\partial_{t}\varphi(t) -\overline{\partial_{t}\varphi}(t))\|^{2}\\
&\quad \le C(e^{-\alpha t}\|\varPsi^{\prime}(\varphi)\|_{L^{1}} +\|\partial_t\varphi\|^2 +\|\nabla\mathcal{N}\partial_t\sigma\|^2) +A|\Omega|\frac{d}{dt} \Big(\overline{\varPsi^{\prime}(\varphi)}\, \overline{\partial_{t}\varphi}\Big),\quad \forall\, t\geq 0,
\label{energyd3}
\end{aligned}
\ee
The initial value of $\mathcal{H}$ is give by
\begin{align}
\mathcal{H}(0)
& =
\frac{\epsilon}{2} \|\partial_{t}\varphi_0 -\overline{\partial_{t}\varphi_0}\|^{2} +\frac12\|\nabla\mathcal{N}(\partial_{t}\varphi_0 -\overline{\partial_{t}\varphi_0})\|^{2}
=: \mathcal{H}_0,\label{iniE01}
\end{align}
where in view of \eqref{reeq}, the initial value for the time derivative of $\varphi$ denoted by $\partial_{t}\varphi_0$ is understood as
\begin{align}
\partial_{t}\varphi_0& = (\partial_{t}\varphi(t))|_{t=0}\notag\\
&=(\epsilon+\mathcal{N})^{-1}\big(B\Delta \varphi_0-A\varPsi^{\prime}(\varphi_0) +\chi\sigma_0 +A\overline{\varPsi^{\prime}(\varphi_0)} -\chi\overline{\sigma_0} -\alpha\mathcal{N}(\varphi_0-\overline{\varphi_0})\big)\notag\\
&\quad -\alpha(\overline{\varphi_0}-c_0).
\end{align}
From our assumption on the initial data, it follows that $\|\partial_{t}\varphi_0\|\leq C$, where $C>0$ may depend on $\|\varphi_0\|_{H^2}$, $\|\sigma_0\|$,  $\delta_0$, $\Omega$ and coefficients of the system. As a result, $\mathcal{H}_0\leq C$. On the other hand, recalling the relation \eqref{varPsi} and the dissipative estimate \eqref{diss2}, we see that
\begin{align}
&\int_0^{+\infty} e^{-\alpha t}\|\varPsi^{\prime}(\varphi(t))\|_{L^{1}}\,dt =\sum_{n=0}^\infty \int_n^{n+1} e^{-\alpha t}\|\varPsi^{\prime}(\varphi(t))\|_{L^{1}}\,dt \notag \\
&\quad \leq \sum_{n=0}^\infty e^{-\alpha n} \int_n^{n+1} \|\varPsi^{\prime}(\varphi(t))\|_{L^{1}}\,dt  \le C,
\label{a3}
\end{align}
where ${C}>0$ is a constant depending on norms of the initial data, $\Omega$ and the coefficients of the system. Thus, integrating \eqref{energyd3} with respect to time and using the estimates \eqref{phit}, \eqref{phit1}, \eqref{a3}, we obtain
\begin{align}
& \mathcal{H}(t)  +\frac{B}{2} \int_0^t \|\nabla \partial_t\varphi(\tau)\|^2\,d\tau \notag\\
&\quad  \leq  C+  A|\Omega|\,\Big(\overline{\varPsi^{\prime}(\varphi(t))}\, \overline{\partial_{t}\varphi(t)}\Big) -A|\Omega|\,\Big(\overline{\varPsi^{\prime}(\varphi_0)}\, \overline{\partial_{t}\varphi_0}\Big) \notag \\
&\quad \leq C+  A|\Omega|\,\Big(\overline{\varPsi^{\prime}(\varphi(t))}\, \overline{\partial_{t}\varphi(t)}\Big), \quad \forall\,t\geq 0.
\label{menergy2}
\end{align}
From the definition of $\mathcal{H}(t)$, it follows that
\be
\|\partial_t\varphi(t)-\overline{\partial_t\varphi(t)}\|^2\le 2\epsilon^{-1}A|\Omega|\overline{\varPsi^{\prime}(\varphi(t))}\, \overline{\partial_{t}\varphi(t)}+C,\quad \forall\, t\in[0,+\infty),\label{e9}
\ee
which further implies
\be
\|\partial_t\varphi(t)\|^2\le 4\epsilon^{-1}A|\Omega|\overline{\varPsi^{\prime}(\varphi(t))}\, \overline{\partial_{t}\varphi(t)} +C,\quad \forall \, t\in[0,+\infty).\label{e7}
\ee

\textbf{Second estimate}.
Testing \eqref{reeq} with $\Delta \varphi$ and using the facts $\int_\Omega \Delta\varphi dx=0$ and  $\varPsi^{\prime\prime}(\varphi)\ge -K$, we obtain
\begin{align}
&B\|\Delta \varphi\|^2+ \alpha\|\varphi -\overline{\varphi}\|^2 \notag\\
&\quad=
(\mathcal{N}(\partial_{t}\varphi -\overline{\partial_{t}{\varphi}}),\Delta \varphi) +A(\varPsi^{\prime}(\varphi), \Delta \varphi) -\chi(\sigma,\Delta \varphi) +\epsilon(\partial_t\varphi,\Delta \varphi)\notag\\
&\quad=
-(\partial_{t}\varphi-\overline{\partial_{t}{\varphi}}, \varphi) - A(\varPsi^{\prime\prime}(\varphi)\nabla \varphi,\nabla \varphi)-\chi(\sigma,\Delta \varphi)+(\epsilon\partial_t\varphi,\Delta \varphi) \notag\\
	&\quad\le \frac{1}{2}\|\partial_t\varphi\|^2 +\frac{1}{2}\|\varphi\|^2 +|\Omega||\overline{\varphi}||\overline{\partial_t\varphi}|
+AK\|\nabla\varphi\|^2 +\frac{B}{2}\|\Delta \varphi\|^2 \notag \\
&\qquad  +\frac{\chi^2}{B}\|\sigma\|^2 +\frac{\epsilon^2}{B}\|\partial_t\varphi\|^2.
\label{lapphi}
\end{align}
From \eqref{aver}, \eqref{prioria20} and \eqref{lapphi}, we deduce that
\be
\begin{aligned}
\frac{B}{2}\|\Delta \varphi\|^2+ \alpha\|\varphi-\overline{\varphi}\|^2 \le \Big(\frac{\epsilon^2}{A}+\frac{1}{2}\Big) \|\partial_t\varphi\|^2+C,
\label{e8}
\end{aligned}
\ee
where ${C}>0$ is a constant depending on norms of the initial data, $\Omega$, the coefficients of the system.

\textbf{Third estimate}. It follows from the equation \eqref{reeq} and the estimates \eqref{prioria20}, \eqref{e7}, \eqref{e8} that
\begin{align}
&\|\varPsi^{\prime}(\varphi) -\overline{\varPsi^{\prime}(\varphi)}\|^2 \notag\\
&\quad=
\frac{1}{A^2}\|B\Delta \varphi +\chi(\sigma-\overline{\sigma}) -\epsilon(\partial_t\varphi -\overline{\partial_{t}{\varphi}}) -\mathcal{N}(\partial_{t}\varphi -\overline{\partial_{t}{\varphi}}) -\alpha\mathcal{N}(\varphi-\overline{\varphi})\|^2 \notag \\
&\quad\le
\frac{C}{A^2}(B^2\|\Delta \varphi\|^2 +(\epsilon^2+C)\|(\partial_t\varphi -\overline{\partial_{t}{\varphi}})\|^2 +\alpha^2\|\mathcal{N}(\varphi-\overline{\varphi})\|^2 +\chi^2\|\sigma-\overline{\sigma}\|^2) \notag \\
&\quad\le C\Big(\frac{B}{2}\|\Delta \varphi\|^2+ \alpha\|\varphi-\overline{\varphi}\|^2\Big) +C\|\partial_t\varphi\|^2+C\notag\\
&\quad \leq C_3(1+\epsilon^2)\overline{\varPsi^{\prime}(\varphi)}\, \overline{\partial_{t}\varphi} + C.
\label{e1}
\end{align}
From the assumption on initial data, we have $\bar{\delta}_0:=1-|\overline{\varphi_0}|\in (0,1]$. On the other hand, by $c_0\in (-1,1)$ and \eqref{aver}, we have
$\overline{\varphi}(t)\in [c_0,\overline{\varphi_0}]\subset [c_0,1-\bar{\delta}_0]$ if $\overline{\varphi_0}\geq c_0$, or $\overline{\varphi}(t)\in [\overline{\varphi_0},c_0)\subset[-1+\bar{\delta}_0,c_0)$ if $\overline{\varphi_0}<c_0$. Hence, we have
\be
|\overline{\varphi}(t)|\leq 1-\delta_1,\quad \forall\,t\geq 0,\quad \text{with}\ \delta_1=\min\{1-c_0,\, c_0+1,\,\bar{\delta}_0\}>0. \label{distm}
\ee
Thanks to \eqref{distm}, we can apply \cite[Proposition A.2]{MZ04} to obtain
\be
\begin{aligned}
|\overline{\varPsi^{\prime}(\varphi)}|
&\le C\|\varPsi^{\prime}(\varphi) -\overline{\varPsi^{\prime}(\varphi)}\|_{L^1} +C,\label{e21}
\end{aligned}
\ee
where the constant $C>0$ depends on $\delta_1$.
From Young's inequality, we have
\be
\begin{aligned}
|\overline{\varPsi^{\prime}(\varphi)}| 	
&\le \frac{1}{2C_3(1+\epsilon^2)}\|\varPsi^{\prime}(\varphi) - \overline{\varPsi^{\prime}(\varphi)}\|^2 +C. \label{e2}
\end{aligned}
\ee
In light of \eqref{e1} and \eqref{e2}, we have
\be
|\overline{\varPsi^{\prime}(\varphi(t))}|\le C,\quad \forall\, t\in [0,+\infty),
\label{e3}
\ee
where ${C}$ is a constant depending on norms of the initial data, $\Omega$, the coefficients of the system.
By \eqref{menergy2}, \eqref{e8}, \eqref{e1} and \eqref{e3}, we obtain that
\be
\|\partial_t\varphi(t)\|^2+\|\Delta \varphi(t)\|^2+\|\varPsi'(\varphi(t))\|^2 +\int_0^t\|\nabla \partial_t\varphi(\tau)\|^2\ d\tau\le C, \quad \forall\, t\in [0,+\infty).
\label{as1}
\ee
The constant ${C}$ may depend on the norm of the initial data, $\Omega$ and coefficients of the system, but is independent on time.
Thus, the above estimate together with \eqref{prioria20} and \eqref{phit} yield
\be
\| \partial_t\varphi\|_{ L^{\infty}(0,+\infty;L^2(\Omega))} +\|  \partial_t\varphi\|_{L^2(0,+\infty;H^1(\Omega))}+\| \varphi\|_{ L^{\infty}(0,+\infty;H^2(\Omega))}\le C.\label{evar}
\ee
From \eqref{e1} and \eqref{e3}, we have
\be
\| \varPsi^{\prime}(\varphi)\|_{ L^{\infty}(0,+\infty;L^2(\Omega))}\le C.\label{evar1}
\ee
Then by the estimates \eqref{prioria20} and \eqref{evar}--\eqref{evar1}, we infer from the equation \eqref{f4.d} that
\be
\| \mu\|_{ L^{\infty}(0,+\infty;L^2(\Omega))}\le C.
\label{evar2}
\ee

\textbf{Fourth estimate}. Multiplying \eqref{f2.b} with $-\Delta\sigma$ and integrating over $\Omega$, we get
\be
\frac{1}{2} \frac{d}{d t}\|\nabla\sigma\|^{2} +\|\Delta \sigma\|^2 = \chi (\Delta \varphi, \Delta\sigma).
\notag
\ee
An application of Young's inequalities leads to
\begin{align*}
\chi (\Delta \varphi,\Delta\sigma)
&\le  \frac12\|\Delta\sigma\|^2 +\frac{\chi^2}{2}\|\Delta \varphi\|^2, \\
 \|\nabla\sigma\|^{2} &=-(\sigma,\Delta\sigma)\le\frac14\|\Delta\sigma\|^2 +\|\sigma\|^2,
\end{align*}
which together with the estimates \eqref{prioria20}, \eqref{evar} imply that
\begin{align}
&\frac{1}{2} \frac{d}{d t}\|\nabla\sigma\|^{2}+\|\nabla\sigma\|^{2}+\frac{1}{4}\|\Delta \sigma\|^2 \leq C,\notag
\end{align}
From Gronwall's lemma and \eqref{prioria20}, \eqref{prioria4}, we obtain
\be
\| \sigma\|_{L^{\infty}(0,+\infty; H^1(\Omega))} \le C,\quad \| \sigma\|_{L^{2}(0,T;H^2(\Omega))}\leq C_T.
\notag
\ee
By comparison with \eqref{f2.b}, we infer from \eqref{evar} and the above estimate that
\be
\partial_t \sigma\in L^{2}_{loc}(0,+\infty;L^2(\Omega)).
\notag
\ee
Next, differentiating \eqref{f2.b} with respect to time, multiplying the resultant by $\partial_t\sigma$ and integrating over $\Omega$, we get
\begin{align}
\frac12 \frac{d}{dt} \|\partial_t\sigma\|^2 + \|\nabla \partial_t\sigma\|^2 \leq \frac12\|\nabla \partial_t \sigma\|^2+\frac{\chi^2}{2}\|\nabla \partial_t\varphi\|^2.
\label{sigH2}
\end{align}
Noticing that
$$
\|\partial_t\sigma(0)\|=\|\Delta\sigma_0-\chi\Delta \varphi_0\|\leq \|\sigma_0\|_{H^2}+|\chi|\|\varphi_0\|_{H^2}
$$
and invoking the estimate \eqref{evar}, we conclude
\begin{align}
\|\partial_t\sigma(t)\|^2+\int_0^t\|\nabla \partial_t\sigma(\tau)\|^2d\tau \leq C,\quad \forall\, t\geq 0.\label{sigtL2}
\end{align}
As $\overline{\partial_t\sigma}=0$, we deduce from the Poincar\'{e}--Wirtinger inequality and \eqref{sigtL2} that
\begin{align}
\|\partial_t\sigma\|_{L^2(0,+\infty;H^1(\Omega))}\le C.\notag
\end{align}
By the elliptic estimate for the Neumann problem and \eqref{evar}, \eqref{sigtL2}, we have
\begin{align}
\|\sigma\|_{L^\infty(0,+\infty;H^2(\Omega))}\le C,\notag
\end{align}
which together with the Sobolev embedding theorem yields
\begin{align}
\|\sigma\|_{L^\infty(0,+\infty;L^\infty(\Omega))}\le C.
\label{sigLi}
\end{align}

\textbf{Fifth estimate}. As in \cite{MZ04}, we rewrite \eqref{reeq} as
\be
\epsilon\partial_t\varphi- B\Delta \varphi +A\varPsi_0^{\prime}(\varphi)=\tilde{h},\label{me1}
\ee
where the right-hand side is give by
\be\tilde{h} = \chi\sigma -\chi\overline{\sigma}  + A\overline{\varPsi^{\prime}(\varphi)} +\epsilon\overline{\partial_{t}\varphi}  -\mathcal{N}(\partial_{t}\varphi -\overline{\partial_{t}\varphi}) -\alpha\mathcal{N}(\varphi-\overline{\varphi}) + A\theta_0\varphi.\notag
\ee
Then, thanks to the estimates \eqref{as1}, \eqref{evar} and  \eqref{sigLi}, we have
\be
\|\tilde{h}\|_{L^{\infty}\left(0, +\infty; L^{\infty}(\Omega)\right)} \leq C_h,
\notag
\ee
where the positive  constant ${C}_h$ may depend on the norm of the initial data, $\Omega$ and coefficients of the system.
The above estimate enables us to invoke the following auxiliary ODE systems as in \cite[Section 3]{MZ04}:
\begin{equation}
\begin{cases}
&\epsilon\displaystyle{\frac{d}{dt}}y_\pm(t) +A\varPsi_0^{\prime}(y_\pm(t))=\pm C_h, \quad \forall\, t\in[0,+\infty),\\
&y_\pm(0)=\pm(1-\delta_0).
\end{cases}
\label{ode1}
\end{equation}
By the Picard--Lindel\"{o}f theorem and the comparison principle, we see that the solutions $y_+$ and $y_-$ are well-defined for $t\geq 0$ and there exists a constant $\delta=\delta(\delta_0,C_h)\in (0,\delta_0]$ satisfying (cf. \cite[Proposition A.3]{MZ04})
\be
y_+(t)\le 1-\delta,\quad y_-(t)\ge -1+\delta,\quad \forall\, t\ge 0. \notag
\ee
Due to the comparison principle for second-order parabolic equations, for \eqref{me1}, we see that
\be
-1+\delta\leq y_-(t)\leq \varphi(t, x) \leq y_+(t)\le 1-\delta, \quad \forall\, t \geq 0, \ x \in \Omega.\notag
\ee
As a consequence, it holds
\be
\|\varphi(t)\|_{L^{\infty}} \leq 1-\delta,\quad \forall\, t \in[0,+\infty),
\label{sep6}
\ee
which implies that, if the initial value of $\varphi(t)$ is strictly separated from the pure states $\pm 1$, then it is uniformly separated from $\pm 1$ for all time.

\textbf{Sixth estimate}. Consider the elliptic problem
\begin{equation}
\begin{cases}
-B\Delta \varphi =\mu -A\varPsi'(\varphi) +\chi\sigma-\epsilon\partial_t \varphi,\quad \text{in}\ \Omega,\\
 {\partial}_{\bm{n}}\varphi=0,\qquad \qquad \qquad \qquad \qquad \qquad  \ \, \text{on}\ \partial\Omega.
\end{cases}\notag
\end{equation}
From (H1), \eqref{mu0}, \eqref{evar} and the standard elliptic estimate for the Neumann problem, we infer  that
\begin{align}
\|\varphi\|_{L^2(0,T;H^3(\Omega))}\leq C_T.\label{h3}
\end{align}
In a similar manner, invoking \eqref{f2.b}, the above estimate \eqref{h3} together with \eqref{sigtL2} yields
\begin{align}
\|\sigma\|_{L^2(0,T;H^3(\Omega))}\leq C_T.\notag
\end{align}

In summary, based on the above \textit{a priori} estimates, we are able to prove the existence of global strong solutions to problem \eqref{f1.a}--\eqref{ini0} by using a suitable approximation scheme (e.g., the one in Section \ref{sgs}) and a standard compactness argument (cf.  \cite{MZ04,MT}). This completes the proof of Theorem \ref{main}.

\section{Long-time Behavior}
\label{ce}
\setcounter{equation}{0}
In this section, we aim to study the long-time behavior of global weak solutions to problem \eqref{f1.a}--\eqref{ini0}.
\subsection{Instantaneous regularity}\label{ss1}
First, we show the instantaneous regularity of weak solutions for $t>0$, in particular, the instantaneous separation from pure states $\pm 1$.

\textbf{Proof of Corollary \ref{cr}}. For any given initial datum $(\varphi_0, \sigma_0)$ satisfying
$\varphi_{0}\in H^1(\Omega)$, $\sigma_{0}\in L^2(\Omega)$ with $\|  \varphi_{0} \|_{L^{\infty}} \le 1$ and $|\overline{\varphi_0}|<1$, let  $(\varphi,\sigma)$ be the corresponding unique global weak solution to problem \eqref{f1.a}--\eqref{ini0} defined by Theorem \ref{cr}. For any $\tau\in (0,1]$, we deduce from Theorem \ref{main1}, in particular the estimates \eqref{prioria4}, \eqref{phiH2} and \eqref{PsiL2} that there exists $\tau_{0} \in(\frac{\tau}{6}, \frac{\tau}{3})$ such that
\be
\varphi(\tau_{0}) \in H^{2}_N(\Omega),\ \sigma(\tau_{0})\in H^1(\Omega),\ \varPsi'(\varphi(\tau_{0}))\in L^2(\Omega),\  \partial_{t}\varphi(\tau_{0})\in L^2(\Omega),\label{varphit4}
\ee
where
\begin{align}
\partial_t\varphi(\tau_0)
&=
(\epsilon+\mathcal{N})^{-1}\big(B\Delta \varphi(\tau_0) -A\varPsi^{\prime}(\varphi(\tau_0))   +A\overline{\varPsi^{\prime}(\varphi(\tau_0))} +\chi\sigma(\tau_0) -\chi\overline{\sigma(\tau_0)}\big) \notag\\ &\quad -\alpha(\epsilon+\mathcal{N})^{-1} \mathcal{N}(\varphi(\tau_0)-\overline{\varphi(\tau_0)})  -\alpha(\overline{\varphi(\tau_0)}-c_0).
\label{varphit4a}
\end{align}
To conclude the result of Corollary \ref{cr}, a natural idea is to prove the existence of a unique global strong solution on $[\tau_0,+\infty)$ with the initial data given by \eqref{varphit4}--\eqref{varphit4a}. However, the regularity stated in \eqref{varphit4} is not sufficient for the application of Theorem \ref{main}. To achieve the goal, we extend the approximate procedure in \cite{MZ04} for the viscous Cahn--Hilliard equation. The proof consists of several steps.

\textbf{Step 1}. For $(\varphi(\tau_0), \sigma(\tau_0))$ given by \eqref{varphit4}, we can find a sequence of approximating functions  $(\varphi_0^{(n)},\sigma_0^{(n)})$ with  $n\in\mathbb{N}^{+}$, such that  $\varphi_0^{(n)}$ stays away from the pure states $\pm 1$ and $\sigma_0^{(n)}\in H^2_N(\Omega)$.
To this end, we take
\be
\varphi^{(n)}_0=\Big(1-\frac{1}{2n}\Big)\varphi(\tau_{0}),\quad \sigma^{(n)}_0=\Big(1+\frac{1}{2n}\mathcal{A}_N\Big)^{-1} \sigma(\tau_0). \notag
\ee
It easily follows that $\sigma_0^{(n)}\in H^2_N(\Omega)$ satisfies $\|\sigma_0^{(n)}\|_{H^1} \to \|\sigma(\tau_0)\|_{H^1}$ as $n \to +\infty$. Next, as in \cite{MZ04}, $\varphi^{(n)}_0\in H^2_N(\Omega)$ satisfies $\|\varphi^{(n)}_0\|_{L^{\infty}}\le 1-\frac{1}{2n}$ and
\begin{align*}
\|\varphi^{(n)}_0\|_{H^2} \to \|\varphi(\tau_0)\|_{H^2}, \quad   \|\varPsi'(\varphi^{(n)}_0)\|\to \|\varPsi'(\varphi(\tau_0))\|
\quad \text{as}\ n\to+\infty.
\end{align*}
Denote
\begin{align}
\varphi^{(n)}_1
&:=(\epsilon+\mathcal{N})^{-1}\big(B\Delta \varphi^{(n)}_0 -A\varPsi^{\prime}(\varphi^{(n)}_0)
+A\overline{\varPsi^{\prime}(\varphi^{(n)}_0)} +\chi\sigma^{(n)}_0 -\chi\overline{\sigma^{(n)}_0}\big) \notag  \\
&\quad -\alpha(\epsilon+\mathcal{N})^{-1} \mathcal{N}(\varphi^{(n)}_0-\overline{\varphi^{(n)}_0})  -\alpha(\overline{\varphi^{(n)}_0}-c_0).
\notag
\end{align}
Thanks to \eqref{varphit4a} and the Lebesgue's dominated convergence theorem, it holds (cf. \cite[pp. 562]{MZ04})
\be
\lim_{n\to+\infty} \|\varphi^{(n)}_1\|= \|\partial_t\varphi(\tau_{0})\|. \notag
\ee

\textbf{Step 2}.
Taking $(\varphi^{(n)}_0,\sigma^{(n)}_0)$ as the initial data, we can now apply Theorem \ref{main} to conclude that there exists a unique strong solution $(\varphi^{(n)},\mu^{(n)},\sigma^{(n)})$ to problem \eqref{f1.a}--\eqref{ini0}, satisfying
\be
\begin{aligned} &\|\partial_t\varphi^{(n)}\|_{L^{\infty}(\tau_{0},+\infty; L^2(\Omega)))\cap L^{2}(\tau_0,+\infty;H^1(\Omega))} +\|\varphi^{(n)} \|_{ L^{\infty}(\tau_{0},+\infty;H^2(\Omega))}\\
&\quad
+\| \mu^{(n)}\|_{ L^{\infty}(\tau_0,+\infty;L^2(\Omega))} +\|\varPsi'(\varphi^{(n)})\|_{L^{\infty}(\tau_{0},+\infty; L^2(\Omega))}\\
&\quad +\|\sigma^{(n)} \|_{L^{\infty}(\tau_{0},+\infty;H^1(\Omega))} \le C,
\label{uni}
\end{aligned}
\ee
and
\be
\|  \mu^{(n)}  \|_{L^{2}(\tau_0,T;H^1(\Omega)) }+\|\partial_t\sigma^{(n)}\|_{L^{2}(\tau_0,T;L^2(\Omega))}+\| \sigma^{(n)}\|_{L^{2}(\tau_0,T;H^2(\Omega))}\leq C_T.\label{sigt1}
\ee
Here, the positive constants $C$, $C_T$ may depend on $\|\varphi_0\|_{H^1}$, $\|\sigma_0\|$, $\overline{\varphi_0}$, $\Omega$, coefficients of the system and $\tau_0$ (and thus on $\tau$), but is independent of $n$ according the approximating property of the initial data shown in Step 1. Of course, the constant $C_T$ may also depend on $T$ with $T\geq \tau_0$.

Now for the approximate solution $\sigma^{(n)}$, we infer from \eqref{sigt1} that there exists $\tau_1\in(\tau_0,\frac{2\tau}{3})$ such that  $\partial_t \sigma^{(n)}(\tau_1)\in L^2(\Omega)$ with
\begin{align}
\|\partial_t \sigma^{(n)}(\tau_1)\| \leq C(\tau,\tau_1),\label{tau1}
\end{align}
where the constant $C(\tau, \tau_1)$ depends on $\tau$ and $\tau_1$, but is independent of $n$. We note that $\tau_1$ can be choosen arbitrary close to $\tau_0$.
For any given $h > 0$, we denote the
difference quotient of a function $f$ by $\partial^h_tf(t)=h^{-1}(f(t + h) - f(t))$. Applying it to \eqref{f2.b}, we get
\be
\frac{d}{dt}\partial_{t}^h \sigma^{(n)} -\Delta \partial_{t}^h \sigma^{(n)}
=-\chi\Delta \partial_{t}^h \varphi^{(n)}.
\label{2tphi}
\ee
Multiplying \eqref{2tphi} by  $\partial_{t}^h \sigma^{(n)}$ and integrating over $\Omega$, it follows that
\begin{equation}
\begin{aligned}
&\frac{1}{2} \frac{d}{d t}\|\partial_{t}^h \sigma^{(n)}\|^{2} +\|\nabla \partial_{t}^h \sigma^{(n)}\|^2 =-(\chi\Delta \partial_{t}^h \varphi^{(n)},\partial_{t}^h \sigma^{(n)}).
\label{3diffsig}
\end{aligned}
\end{equation}
Like in Section 4, from Young's inequality, we have
\begin{align}
&\frac{1}{2}\frac{d}{dt} \|\partial_{t}^h\sigma^{(n)}\|^2
 +\frac{1}{2}\|\nabla\partial_{t}^h\sigma^{(n)}\|^2
\le \frac{\chi^2}{2}\|\nabla\partial_{t}^h \varphi^{(n)}\|^2.
\label{3diffsig1}
\end{align}
Recalling the estimates \eqref{uni} and \eqref{tau1}, we can integrate \eqref{3diffsig1} with respect to time on $(\tau_1,+\infty)$, then pass to the limit as $h\to 0$ to get
$$
\|\partial_{t}\sigma^{(n)}\|_{L^{\infty}\left(\tau_1, +\infty ; L^{2}(\Omega) \right) }+\|\nabla\partial_{t}\sigma^{(n)}\|_{L^{2}\left(\tau_1, +\infty ; L^{2}(\Omega) \right)}\leq C.
$$
Hence, it follows from Poincar\'{e}--Wirtinger inequality and the fact $\langle\partial_t \sigma^{(n)} \rangle=0$ that
\be
\|\partial_{t}\sigma^{(n)}\|_{L^2\left(\tau_1, +\infty ; H^{1}(\Omega) \right)}\leq C.\label{sigt4}
\ee
Besides, from the equation \eqref{f2.b} and the elliptic estimate for Neumann problem, we have
\be
\|\sigma^{(n)}\|_{L^{\infty}\left(\tau_1, +\infty ; H^{2}(\Omega)\right)}
\leq C,\label{sig}
\ee
which implies
\be
\|\sigma^{(n)}\|_{L^{\infty}\left(\tau_1, +\infty ; L^{\infty}(\Omega)\right)}\leq C,\label{sig2}
\ee
The constant $C>0$ in \eqref{sigt4}--\eqref{sig2}  depends on $\tau$ and $\tau_1$, but is independent of $n$.

\textbf{Step 3}.
For the equation
\be
\epsilon\partial_t\varphi^{(n)}- B\Delta \varphi^{(n)} +A\varPsi_0^{\prime}(\varphi^{(n)})=\tilde{h}^{(n)},
\label{me1a}
\ee
with its right-hand side given by
\begin{align}
\tilde{h}^{(n)}
&= \chi\sigma^{(n)} -\chi\overline{\sigma^{(n)}}  + A\overline{\varPsi^{\prime}(\varphi^{(n)})} +\epsilon\overline{\partial_{t}\varphi^{(n)}}  -\mathcal{N}(\partial_{t}\varphi^{(n)} -\overline{\partial_{t}\varphi^{(n)}}) \notag \\
&\quad -\alpha\mathcal{N}(\varphi^{(n)}-\overline{\varphi^{(n)}}) + A\theta_0\varphi^{(n)}.\label{mea}
\end{align}
Thanks to estimates \eqref{uni} and \eqref{sig2}, we have
\be
\|\tilde{h}^{(n)}\|_{L^{\infty}\left(\tau_1, +\infty; L^{\infty}(\Omega)\right)} \leq \tilde{C}_h,
\label{asi}
\ee
where the constant $\tilde{C}_h>0$ is independent of $n$. Consider the initial value problem of ODEs (cf. \eqref{ode1})
\begin{equation}
\begin{cases}
&\epsilon\displaystyle{\frac{d}{dt}}y_\pm^{(n)}(t) +A\varPsi_0^{\prime}(y_\pm^{(n)}(t))=\pm \tilde{C}_h, \quad \forall\, t\in[\tau_1,+\infty),\\
&\\
&y^{(n)}_\pm(\tau_1)=\pm(1-\frac{1}{2n}).
\end{cases}
\label{ode2}
\end{equation}
We infer from \cite[Corollary A.1]{MZ04} that for any
$\tau_2>\tau_1$, there exists a constant $\tilde{\delta}\in (0,1)$ depending on $\tau$, $\tau_1$,  $\tilde{C}_h$ and $\tau_2$, but is independent of $n$, such that
\be
|y^{(n)}_\pm(t)|\le 1-\tilde{\delta}, \quad \forall\, t\geq \tau_2.\notag
\ee
Then by the comparison principle for second-order parabolic equations, we obtain
\be
-1+\tilde{\delta} \leq y^{(n)}_-(t)\leq \varphi^{(n)}(t, x) \leq y_+^{(n)}(t)\le 1-\tilde{\delta}, \quad \forall\, t \geq \tau_2, \ x \in \Omega,\label{sp5a}
\ee
which implies the uniform separation of $\varphi^{(n)}$ from the pure states $\pm 1$ for $t\geq \tau_2$. We note that $\tilde{\delta}$ may depend on $\tau$, $\tau_1$ and $\tau_2$, but is independent of $n$. From \eqref{sp5a} and by a similar argument in the sixth estimate in Section 4, we also obtain
\begin{align}
\|\varphi^{(n)}\|_{L^2(\tau_2,T;H^3(\Omega))}+ \|\sigma^{(n)}\|_{L^2(\tau_2,T;H^3(\Omega))}\leq C_T.
\label{sp6}
\end{align}

\textbf{Step 4}.
Since the estimates \eqref{uni}, \eqref{sigt1}, \eqref{sigt4}--\eqref{sig2}, \eqref{sp5a} and \eqref{sp6} obtained in the previous steps are independent of $n$, then by using a similar compactness argument like in \cite[Section 4]{MT}, we are able to pass to the limit as $n \to+ \infty$ and find a convergent subsequence of $\{(\varphi^{(n)}, \sigma^{(n)},\mu^{(n)})\}$ such that the limit function denoted by $(\hat{\varphi}, \hat{\sigma}, \hat{\mu})$ is indeed a global weak solution to problem \eqref{f1.a}--\eqref{ini0} on $[\tau_0,+\infty)$  subject to the initial datum
$(\varphi(\tau_0), \sigma(\tau_0))$,
with the following regularity properties
\be
\begin{aligned}
    &\hat{\varphi} \in  L^\infty(\tau_0,+\infty ;H^2_N(\Omega)), \\
	&\partial_t\hat{\varphi} \in L^{\infty}(\tau_0,+\infty;L^2(\Omega))\cap L^{2}(\tau_0,+\infty;H^1(\Omega)), \\
	&\hat{\mu} \in   L^{\infty}(\tau_0,+\infty;L^2(\Omega))\cap  L^{2}_{loc}(\tau_0,+\infty;H^1(\Omega)), \\
	&\hat \sigma  \in L^\infty(\tau_0,+\infty;H^1(\Omega))\cap L^{2}_{loc}(\tau_0,+\infty;H^2(\Omega)),\\
&\partial_t\hat \sigma  \in L^{2}_{loc}(\tau_0,+\infty;H^1(\Omega)),
\end{aligned}\notag
\ee
as well as
\begin{align*}
&\hat \sigma \in L^\infty(\tau_1,+\infty; H^2_N(\Omega)),\quad  \partial_t\hat \sigma \in L^2(\tau_1,+\infty; H^1(\Omega),\\
& \hat\varphi\in L^2_{loc}(\tau_2,+\infty; H^3(\Omega)),\quad \hat\sigma\in L^2_{loc}(\tau_2,+\infty; H^3(\Omega)),
\end{align*}
for $\tau_0<\tau_1<\tau_2<\tau$, and in particular,
\be
\|\hat\varphi(t)\|_{L^\infty}\leq 1-\tilde{\delta}, \quad \forall\, t \geq \tau_2.\notag
\ee

Hence, we are now in a position to complete the proof of Corollary \ref{cr}. Thanks to the uniqueness of weak solutions of problem \eqref{f1.a}--\eqref{ini0}, it follows that
$(\varphi, \mu, \sigma)(t)=(\hat{\varphi}, \hat{\sigma}, \hat{\mu})(t)$ for all $t\geq \tau_0$.
Hence, the global weak solution $(\varphi, \mu, \sigma)$  becomes a strong one on the interval $[\tau,+\infty)$. Besides, fixing $\tau_1\in (\frac{\tau}{2}, \frac{2\tau}{3})$ and $\tau_2=\frac{5\tau}{6}$ in the estimates of previous steps, we can obtain the uniform separation property
 \be
\|\varphi(t)\|_{L^\infty}\leq 1-\delta_\tau, \quad \forall\, t \geq \tau.\label{sp5c}
\ee
 where the constant $\delta_\tau\in (0,1)$ depends on $\tau$, $\|\varphi_0\|_{H^1}$, $\|\sigma_0\|$, $1-|\overline{\varphi_0}|$, $\Omega$ and coefficients of the system. The proof of Corollary \ref{cr} is complete.
$\hfill\square$

\subsection{Existence of a global attractor}

Now we prove the existence of a global attractor, which is the unique compact set in a suitable phase space, being invariant under the semigroup generated by the evolution problem and attracting all bounded sets as time goes to infinity (see \cite{T}). \smallskip

\textbf{Proof of Theorem \ref{attr}}.
First, it follows from Theorem \ref{main1} that the global weak solutions to problem \eqref{f1.a}--\eqref{ini0} defines a closed semigroup $\mathcal{S}(t)$ on the phase space $\mathcal{X}_{m_1,m_2}$ in the sense of \cite{PZ07}. Next, we observe that Corollary \ref{cr} implies the asymptotic compactness of $\mathcal{S}(t)$. In particular, let $(\varphi, \sigma)$ be a global weak solution to problem \eqref{f1.a}--\eqref{ini0}, we infer that for any $\tau >0$,
$$
\|\varphi(t)\|_{H^2}+\|\sigma(t)\|_{H^2}\leq C_\tau,\quad \forall\,t\geq \tau,
$$
where the constant $C_\tau>0$ depends on $\|\varphi_0\|_{H^1}$, $\|\sigma_0\|$, $\Omega$, coefficients of the system and $\tau$, but is independent of $t$. From the above estimate, the continuous dependence estimate \eqref{conti2} and a standard interpolation argument (cf. e.g., \cite[Proposition 6.1]{GG}), we see that $\mathcal{S}(t)\in C(\mathcal{X}_{m_1,m_2}, \mathcal{X}_{m_1,m_2})$ for all $t\geq 0$. As a consequence, $\mathcal{S}(t)$ is actually a strongly continuous
semigroup on $\mathcal{X}_{m_1,m_2}$. Finally, we note that the estimates \eqref{aver}, \eqref{1aver}, \eqref{phi} together with the dissipative estimate  \eqref{diss1} implies the dissipativity property of $\mathcal{S}(t)$ in $\mathcal{X}_{m_1,m_2}$. Namely,
 $\mathcal{S}(t)$ possesses a bounded absorbing ball $\mathcal{B}_0 \subset \mathcal{X}_{m_1,m_2}$, whose radius depends on $\Omega$, $m_1$, $m_2$ and coefficients of the system but is independent of the initial data. For
every bounded set $\mathcal{B} \subset \mathcal{X}_{m_1,m_2}$, there exists a time $t_0 = t_0(\mathcal{B}) > 0$ such that $\mathcal{S}(t)\mathcal{B}\subset \mathcal{B}_0$ for all $t\geq t_0$.

Thanks to the above facts, by applying a standard argument in the theory of infinite dimensional dynamical systems \cite{T}, we can conclude that the dynamical system  $(\mathcal{S}(t),\mathcal{X}_{m_1,m_2})$ defined by problem \eqref{f1.a}--\eqref{ini0} admits a global attractor $\mathcal{A}_{m_1,m_2}\subset \mathcal{X}_{m_1,m_2}$ that is bounded in $H^2(\Omega)\times H^2(\Omega)$. The proof of Theorem \ref{attr} is complete. \hfill $\square$

\subsection{The $\omega$-limit set}

Finally, we proceed to characterize the $\omega$-limit set of an arbitrary given initial datum $(\varphi_0, \sigma_0)$ belonging to the set
$$
\mathcal{Z}=\left\{(z_1,z_2) \in H^{1}(\Omega)\times L^{2}(\Omega):\|  z_1 \|_{L^{\infty}} \le 1,\
|\overline{ z_1}|<1\right\}.
$$
Denoting $\mathcal{H}=H^{2r}(\Omega)\times H^{2r}(\Omega)$ with $r\in (1/2,1)$, we define the $\omega$-limit set of  $(\varphi_{0},\sigma_0)$ as
\begin{align}
\omega\left(\varphi_{0},\sigma_0\right)
&=\big\{(z_1,z_2) \in \mathcal{Z}: \exists\,\left\{t_{n}\right\} \nearrow +\infty\  \text { s.t. } \left(\varphi\left(t_{n}\right),\sigma \left(t_{n}\right)\right) \rightarrow (z_1,z_2)\ \text{ strongly in } \mathcal{H} \big\},
\label{df}
\end{align}
where $(\varphi, \sigma)$ is the unique global weak solution corresponding to $(\varphi_0,\sigma_0)$. Then we have

\begin{proposition} \label{sta}
Suppose that (H1)--(H2) are satisfied.
For any initial datum $(\varphi_{0},\sigma_0) \in \mathcal{Z}$, its $\omega$-limit set $\omega(\varphi_{0},\sigma_0)$ is  nonempty and compact in $\mathcal{H}$. Moreover,

(1) $\omega(\varphi_0,\sigma_0)$ consists of stationary points $(\varphi_{\infty}, \sigma_{\infty})$ only, where $(\varphi_{\infty}, \sigma_{\infty}) \in H^2_N(\Omega)\times H^2_N(\Omega)$ is a strong solution to the stationary problem \eqref{5bchv}--\eqref{5dchv}, satisfying $(\overline{\varphi_\infty},\overline{\sigma_\infty}) =(c_0,\overline{\sigma_0})$. In particular, there exists a constant $\delta\in (0,1)$ such that
	\be
	\|\varphi_{\infty}\|_{L^{\infty}}\le 1-\delta,\label{ps1}
	\ee
where $\delta$ is independent of $\varphi_\infty$.

(2) $\omega(\varphi_{0},\sigma_0)$ is invariant under the semigroup $\mathcal{S}(t)$ defined by the global weak solution to problem \eqref{f1.a}--\eqref{ini0}, that is, $\mathcal{S}(t)\omega(\varphi_{0},\sigma_0) =\omega(\varphi_{0},\sigma_0)$ for all $t\geq 0$.
\end{proposition}

\begin{proof}
For any initial datum $(\varphi_0, \sigma_0)\in \mathcal{Z}$, we denote the associated unique global weak solution by $(\varphi, \sigma)$.
From Corollary \ref{cr}, we see that for any fixed $\tau>0$, it holds
 $\varphi \in C([\tau,+\infty); H^{2}(\Omega))$, $\sigma \in C([\tau,+\infty); H^{2}(\Omega))$, being uniformly bounded for all $t\geq \tau$. Hence, it follows from the Sobolev embedding theorem that the orbit $\{\left(\varphi(t),\sigma(t)\right)\}_{t \geq \tau}$ is relatively compact in $\mathcal{H}$, which implies that $\omega\left(\varphi_{0},\sigma_0\right)$ is a nonempty, bounded subset in $H^{2}(\Omega)\times H^{2}(\Omega)$ and thus compact in $\mathcal{H}$.

Then for any cluster point $(\varphi_\infty,\sigma_\infty) \in \omega\left(\varphi_{0},\sigma_0\right)$, we see that $(\varphi_\infty,\sigma_\infty) \in H^{2}(\Omega)\times H^{2}(\Omega)$ and $\overline{ \varphi_\infty}=c_0$, $\overline{\sigma_\infty} =\overline{ \sigma_0}$. Following the argument in \cite{CGW14,HP}, we are able to show that $(\varphi_\infty,\sigma_\infty)$ is indeed a solution to the stationary problem \eqref{5bchv}--\eqref{5dchv}.
To this end, we recall \eqref{i6}, which implies that $\mathcal{F}(t) +\frac{C}{\alpha} e^{-\alpha t}$ is indeed a strict Lyapunov function for problem \eqref{f1.a}--\eqref{ini0}. Since it is non-increasing in time and bounded from below, there exists some constant $\mathcal{F}_{\infty}\in\mathbb{R}$ such that
\be
\lim _{t \rightarrow+\infty}{\mathcal{F}}(t)=\lim _{t \rightarrow+\infty}\left({\mathcal{F}}(t)+\frac{C}{\alpha} e^{-\alpha t}\right)
=\mathcal{F}_{\infty}.
\label{m1}
\ee
Let $\{t_{n}\}$ be an unbounded increasing sequence such that
\begin{align}
(\varphi(t_n),\sigma(t_n)) \rightarrow (\varphi_{\infty},\sigma_\infty) &\quad  \text { strongly in } \mathcal{H},\notag\\
(\varphi(t_n),\sigma(t_n)) \rightarrow (\varphi_{\infty},\sigma_\infty)  &\quad  \text { weakly in } H^{2}(\Omega)\times H^2(\Omega).\notag
\end{align}
For any $\tau\in(0,+\infty)$, we recall Corollary \ref{cr} such that
\be
\|\varphi(t)\|_{L^{\infty}}\le 1-\delta(\tau), \quad \forall\, t\in [\tau,+\infty), \label{sepp}
\ee
which together with the Sobolev embedding theorem $H^{2r}\hookrightarrow L^\infty$ (for $r\in (1/2,1)$) yields the conclusion \eqref{ps1}. Without loss of generality, we assume $t_{n+1} \geq t_{n}+1$ for $n \in \mathbb{N}$. Integrating the inequality \eqref{i6} on the time interval $\left[t_{n}, t_{n+1}\right]$, we obtain
\be
\begin{aligned}
&\mathcal{F}\left(t_{n+1}\right)
-\mathcal{F}(t_{n}) +\frac12\int_{t_{n}}^{t_{n+1}}\|\nabla \tilde{\mu}(s)\|^{2} +\|\nabla (\sigma(s)-\chi\varphi(s))\|^2 +\epsilon\|\partial_t\varphi(s)\|^2\, d s \\
&\quad\leq \frac{C}{\alpha} e^{-\alpha t_n}-\frac{C}{\alpha} e^{-\alpha t_{n+1}},
\label{eq}
\end{aligned}
\ee
which implies
\be
\lim_{n\to+\infty}\int_{t_{n}}^{t_{n+1}}\|\nabla \tilde{\mu}(s)\|^{2} +\|\partial_t\sigma(s)\|^2_{V_0'} +\epsilon\|\partial_t\varphi(s)\|^2\, d s=0.
\label{m3}
\ee
As a result, we have for $n\to\infty$
\begin{align*}
&\|\varphi\left(t_{n}+s_{1}\right) -\varphi\left(t_{n}+s_{2}\right)\| \rightarrow 0, \quad\ \ \, \text { uniformly for all } s_{1}, s_{2} \in[0,1],\\
&\|\sigma\left(t_{n}+s_{1}\right) -\sigma\left(t_{n}+s_{2}\right)\|_{V_0'} \rightarrow 0, \quad \text { uniformly for all } s_{1}, s_{2} \in[0,1].
\end{align*}
Combining the above facts with the boundedness of $(\varphi(t),\sigma(t))$ in $H^2\times H^2$, by a standard interpolation we infer that
\begin{equation}
\|(\varphi (t_{n}+s),\sigma(t_{n}+s)) -(\varphi_\infty,\sigma_\infty)\|_{\mathcal{H}} \rightarrow 0, \quad \text { uniformly for all } s  \in[0,1].
\label{si2}
\end{equation}
Moreover, the strict separation property \eqref{sepp} enables us to regard the singular potential $\varPsi$ as a globally Lipschitz function when $n$ is large. Denote
$\tilde{t}_n:=t_{n}+s$ with $s\in[0,1]$.
Then for any $\xi \in H^1(\Omega)$, we deduce from Lebesgue's dominated convergence theorem that
\begin{align}
&\int_{\Omega}B\nabla \varphi_\infty \cdot \nabla \xi +A\varPsi^{\prime}(\varphi_\infty) \xi -A\overline{\varPsi^{\prime}(\varphi_\infty)}\xi
-\chi \sigma_\infty\xi +\chi\overline{\sigma_\infty}\xi +\alpha \mathcal{N}(\varphi_\infty -\overline{\varphi_\infty})\xi\ d x \notag\\
&\quad=
\lim _{n \rightarrow+\infty}
\int_{0}^{1} \int_{\Omega} B\nabla \varphi(\tilde{t}_n) \cdot \nabla \xi +A\varPsi^{\prime}(\varphi(\tilde{t}_n)) \xi -A\overline{\varPsi^{\prime}(\varphi(\tilde{t}_n))}\,\xi \notag\\
&\qquad\qquad\qquad\qquad
-\chi \sigma(\tilde{t}_n)\xi +\chi\overline{\sigma(\tilde{t}_n)}\xi
+\alpha \mathcal{N}(\varphi(\tilde{t}_n)
-\overline{ \varphi(\tilde{t}_n)}\,)\, \xi\, d x\, d s \notag\\
&\quad
= \lim _{n \rightarrow+\infty}
\int_{0}^{1} \int_{\Omega} \big( \tilde{\mu}(\tilde{t}_n) -\overline{\tilde{\mu}(\tilde{t}_n)}
-\epsilon\partial_t\varphi(\tilde{t}_n) +\epsilon\overline{\partial_t{\varphi}(\tilde{t}_n)}\,\big) \xi\, d x\, d s \notag \\
&\quad \leq
\lim _{n \rightarrow +\infty} \int_{0}^{1} \|\tilde{\mu}(\tilde{t}_n) -\overline{\tilde{\mu}(\tilde{t}_n)}\| \|\xi\|\, d s +\epsilon \int_{0}^{1} \|\partial_t\varphi (\tilde{t}_n)-\overline{\partial_t\varphi(\tilde{t}_n)}\| \|\xi\|\, d s \notag \\
&\quad \leq \lim _{n \rightarrow+\infty} \left(\int_{0}^{1}\|\nabla \tilde{\mu}(\tilde{t}_n)\|^{2}\, d s\right)^{\frac{1}{2}}\|\xi\|  + C\epsilon \left(\int_{0}^{1}\|\partial_t\varphi(\tilde{t}_n)\|^{2} d s\right)^{\frac{1}{2}}\|\xi\| \notag\\
&\quad=0 .\notag
\end{align}
From \eqref{eq}, we also infer that
\be
\int_{\Omega}\nabla (\sigma_\infty -\chi\varphi_\infty)\cdot \nabla\xi\,d x=0,\quad \forall\,\xi \in H^1(\Omega). \notag
\ee
As a consequence, $(\varphi_\infty,\sigma_\infty)$ is a weak solution to the stationary problem \eqref{5bchv}--\eqref{5dchv} with $\mu_\infty =A\overline{\varPsi^{\prime}(\varphi_{\infty})} -\chi\overline{\sigma_{\infty}}$. Since we already know that $(\varphi_\infty,\sigma_\infty)\in H^2(\Omega)\times H^2(\Omega)$, then it is also a strong solution.

Finally, since every $(\varphi_\infty,\sigma_\infty)\in \omega(\varphi_0,\sigma_0)$ is a stationary point, i.e., $\mathcal{S}(t)(\varphi_\infty,\sigma_\infty) =(\varphi_\infty,\sigma_\infty)$ for $t\geq 0$, this yields the invariance of $\omega(\varphi_0, \sigma_0)$ under $\mathcal{S}(t)$. The proof is complete.
\end{proof}
\smallskip

\textbf{Proof of Theorem \ref{main4}}. It is straightforward to check that the conclusion of Theorem \ref{main4} is an immediate consequence of Proposition \ref{sta}. \hfill $\square$ \smallskip

We note that Proposition \ref{sta} also provides a dynamical approach for the study of the stationary problem \eqref{5bchv}--\eqref{5dchv}, which is independent of the viscous parameter $\epsilon$.

\begin{corollary}\label{sta2}
Let $(\varphi_{*},\sigma_{*})\in \mathcal{Z}$  be a weak solution of the stationary problem \eqref{5bchv}--\eqref{5dchv}. Then $(\varphi_{*},\sigma_{*})\in H^2_N(\Omega)\times H^2_N(\Omega)$ and there exists a positive constant $\delta\in (0,1)$ such that the strict separation property \eqref{ps1} holds for $\varphi_*$.
\end{corollary}
\begin{proof}
It is obvious that $(\varphi_{*},\sigma_{*})$ can be viewed as a global weak solution to the evolution problem \eqref{f1.a}--\eqref{ini0} with the particular choice of initial datum $(\varphi_0, \sigma_0)=(\varphi_{*}, \sigma_{*})$ (cf. \cite{MZ04} for the viscous Cahn--Hilliard equation). Thanks to the uniqueness of weak solutions (recall \eqref{conti2}), we have $\mathcal{S}(t)(\varphi_{*},\sigma_{*}) =(\varphi_{*},\sigma_{*})$ for all $t\geq 0$.
Therefore, we can infer that $(\varphi_{*},\sigma_{*}) \in \omega(\varphi_{*},\sigma_{*})$ and then the conclusion follows from Proposition \ref{sta}.
\end{proof}
\medskip


\noindent
\textbf{Acknowledgments}. \noindent The author would like to thank Professor Hao Wu for helpful discussions.


\end{document}